\pdfoutput=1
\RequirePackage{ifpdf}
\ifpdf % We are running pdfTeX in pdf mode
\documentclass[pdftex]{ClasePaper}
\else
\documentclass{ClasePaper}
\fi

\usepackage[Symbol]{upgreek}
\usepackage[mathscr]{eucal} % Fuente Euscript, Euler Script font.
\usepackage{indentfirst} % Sangria al iniciar Secciones, etc.

\makeatletter
\newcommand{\Xbar}{{\mathchoice
     {\smash@bar\textfont\displaystyle{0.55}{2.5}\mathscr{X}}
     {\smash@bar\textfont\textstyle{0.55}{2.5}\mathscr{X}}
     {\smash@bar\scriptfont\scriptstyle{0.55}{2.5}\mathscr{X}}
     {\smash@bar\scriptscriptfont\scriptscriptstyle{0.55}{2.5}\mathscr{X}}
          }}
\newcommand{\smash@bar}[4]{%
     \smash{\rlap{\raisebox{-#3\fontdimen5#10}{$\m@th#2\mkern#4mu\mathchar'26$}}}%
          }
\makeatother
\newcommand{\X}[1]{\mathscr{\Xbar}_{#1}}

\newcommand{\T}{\mathsf{T}}
\newcommand{\V}{\mathbb{V}}
\newcommand{\HH}{\mathbb{H}}
\newcommand{\dd}{\mathrm{d}}
\newcommand{\OmegaH}{\Omega^{\mathrm{H}}}
\newcommand{\OmegaV}{\Omega^{\mathrm{V}}}
\newcommand{\QH}{\mathrm{Q}_{\mathrm{H}}}
\newcommand{\QV}{\mathrm{Q}_{\mathrm{V}}}
\newcommand{\Curv}{\mathrm{Curv}}
\newcommand{\hor}{\mathrm{hor}}
\newcommand{\dlie}[1]{\mathscr{L}_{#1}}
\newcommand{\cSch}[1]{[\hspace{-0.065cm}[ #1 ]\hspace{-0.065cm}]}
\newcommand{\Cinf}[1]{\mathbf{\mathit{C}}^{\infty}_{#1}}
\DeclareMathOperator{\rank}{rank}

\numberwithin{equation}{section}

\usepackage{multicol}

\begin{document}

\allowdisplaybreaks

%\renewcommand{\PaperNumber}{096}
%\FirstPageHeading

\thispagestyle{empty}

     \begin{center}
\ArticleName{Compatible Poisson Structures on Fibered $\mathbf{5}$-Manifolds}
     \end{center}

\ShortArticleName{Compatible Poisson Structures on Fibered $5$-Manifolds}

\AuthorNameForHeading{R. Flores-Espinoza, J. C. Ru\'iz-Pantale\'on and Yu. Vorobiev}

     \begin{center}
\Author{R. FLORES-ESPINOZA\,$^{1}$, J. C. RU\'IZ-PANTALE\'ON\,$^{2}$ and YU. VOROBIEV\,$^{3}$}

\Address{${\scriptstyle 1}$\ Department of Mathematics, University of Sonora, M\'exico}
\EmailD{\href{mailto:rflorese@mat.uson.mx}{rflorese@mat.uson.mx}}
%\URLaddressD{\url{http://galia.fc.uaslp.mx/~jvallejo/}}

\Address{${\scriptstyle 2}$\ Department of Mathematics, University of Sonora, M\'exico}
\EmailD{\href{mailto:jcpanta@mat.uson.mx}{jcpanta@mat.uson.mx}}

\Address{${\scriptstyle 3}$\ Department of Mathematics, University of Sonora, M\'exico}
\EmailD{\href{mailto:yurimv@guaymas.uson.mx}{yurimv@guaymas.uson.mx}}
     \end{center}

%\ArticleDates{Received May 19, 2014, in f\/inal form September 09, 2014; Published online September 15, 2014}

\Abstract{We study a class of Poisson tensors on a fibered manifold which are compatible with the fiber bundle structure by the so-called almost coupling condition. In the case of a $5$-dimensional orientable fibered manifolds with $2$-dimensional bases, we describe a global behavior of almost coupling Poisson tensors and their singularities by using a bigraded factorization of the Jacobi identity. In particular, we present some unimodularity criteria and describe a class of gauge type transformations preserving the almost coupling property.}

\Keywords{Poisson structures, fiber bundles, almost coupling tensors, Poisson connections}

\Classification{53D17; 53C12; 70G45.}

\vspace{-2mm}

\section{Introduction}

Recall that a Poisson manifold consists of a smooth manifold $M$ equipped with a Poisson structure, that is, a Lie bracket $\{,\}$ on the space of smooth functions $\Cinf{M}$ which is compatible with the pointwise product by the Leibniz rule. A Poisson structure can be given by a bivector field \,$\Pi \in \Gamma\wedge^{2}\T{M}$,\, called a \textit{Poisson tensor}, by the formula \,$\{f,g\}=\Pi(\dd{f},\dd{g})$.\, The Jacobi identity for the bracket $\{,\}$ is equivalent to the nonlinear equation \,$\cSch{\Pi,\Pi}=0$\, for the Schouten bracket. A \textit{singular point} $p$ of the Poisson structure is characterized by the condition that $\rank{\Pi}$ is not constant around $p$. The Hamiltonian vector fields \,$X_{f}=\mathbf{i}_{\dd{f}}\Pi$,\, $f \in \Cinf{M}$\, span singular integrable distribution which induces a partition of $M$ into immersed submanifolds called \textit{symplectic leaves}.

We are interested in Poisson manifolds equipped with additional structures, namely, fiber bundle structures. Fibered manifolds appear as natural phase spaces for various physical models as well in some problems in Poisson geometry (see, for example, \cite{MaMoRa-90,MMR-90,BraFe-08,CrMa-10,MYu,JVYu}).

In this article, we study a class of Poisson tensors $\Pi$ on a fibered manifold $M$ over a base $B$ which are compatible with the fiber bundle structure by the condition: the mixed term of $\Pi$ in the bigraded decomposition relative to an Ehresmann connection (a horizontal subbundle) vanishes and $\Pi$ is the sum of the horizontal and vertical bivector fields on $M$, \,$\Pi=\Pi_{H}+\Pi_{V}$.\, Such bivector fields on $M$ are called the \textit{almost coupling Poisson structures} \cite{Va-04} and generalize the class of \textit{coupling Poisson structures} naturally arising in the study of Poisson geometry around symplectic leaves \cite{YuV-2001}. A coupling Poisson tensor $\Pi$ on a fibered manifold $M$ is characterized by a horizontal nondegeneracy condition and has the following feature: the Jacobi identity for $\Pi$ admits a bigraded factorization which leads to the four equations called the integrability conditions having a natural geometric interpretation. Our goal is to investigate the global behavior of compatible Poisson tensors on fibered manifolds in the almost coupling case with further applications to the linearization problem and normal forms \cite{YuV-2001,CrMa-10,Marcut} and the method of averaging \cite{MYu,JVYu}. In particular, one of our motivations comes from the question on the existence of almost coupling neighborhoods for Poisson submanifolds. The answer to this question is positive \cite{YuV-2001} in the case of symplectic leaves which present a ``simplest'' type of Poisson submanifolds.

Our approach is based on the following key observation. The rank of the horizontal component $\Pi_{H}$ of a bivector field $\Pi$ in a bigraded decomposition is independent of the choice of an Ehresmann connection. This gives rise to an intrinsic decomposition of the
fibered manifold \,$M=\cup_{i=0}^{k}M_{i}$\, (disjoint union), where \,$k=\dim{B}$\, and $M_{i}$ is the subset of all points at which the rank of the horizontal tensor $\Pi_{H}$ equals \,$i\leq k$.\, In particular, \,$M^{\Pi}:=M_{k}$\, is an open subset in $M$, consisting of the points of maximal horizontal rank of $\Pi$. A given almost coupling Poisson tensor $\Pi$ has a ``good'' behavior on
$M^{\Pi}$ in the sense that the restriction $\Pi|_{M^{\Pi}}$ is a coupling tensor and inherits a unique Poisson connection. The question is to describe how the tensor field $\Pi$ in the complement \,$M \setminus M^{\Pi}$\, is glued with its coupling part by passing through the boundary $\partial M^{\Pi}$ which consists of singular points of $\Pi$. We address this question to the case of almost coupling Poisson tensors on $5$-dimensional fibered manifolds with $2$-dimensional bases. Under some orientability assumptions, we give a
complete characterization of such ``toy'' Poisson models by using the bigraded tensorial calculus and the Ehresmann connection technique. In particular, we present some unimodularity criteria and describe a class of symmetries of the integrability conditions by using the gauge type transformations for Poisson structures \cite{SeWe-01,BuRa-03}. Moreover, we illustrate these results by considering the trivial fiber bundles and constructing deformed Poisson structures on the product $(2+3)$-manifolds.

Note that various results on low dimensional Poisson manifolds in dimensions $2$, $3$ and $4$ were obtained, for example, in \cite{LiuXU-92,GMP,GN-93,CIMP-94,Ra-2002,Nar-2015,PabSua-2018}. The results, presented in this paper, can be used for the classification of Poisson structures around $2$-dimensional Poisson submanifolds.

The paper is organized as follows. In Section $2$, we briefly recall the basic facts about Ehresmann connections on fiber bundles. In Section $3$, some general properties of almost coupling Poisson structures are formulated. In Section $4$, we present our main results and give a complete description of almost coupling Poisson tensors on $5$-dimensional orientable fibered manifolds with $2$-dimensional bases. Section $5$ contains a global criterion of unimodularity which extends the results of \cite{AEYu-17} to the almost coupling case. In Section $6$, we describe a method of construction of almost coupling Poisson tensors by using the gauge transformations. In Section $7$, we present some coordinate formulas and illustrate the general results in the case of trivial bundles.

\section{Preliminaries}

Here we recall some basic facts on Ehresmann connections on fiber bundles which will be used in our bigraded calculus on fibered manifolds (for more details, see also \cite{MMR-90,YuV-2001,Va-04,AEYu-17}).

Let \,$\pi:M \rightarrow B$\, be a fiber bundle (a surjective submersion). Denote by \,$\V:=\ker\dd{\pi} \subset \T{M}$\, the \emph{vertical subbundle} and by \,$\V^{\circ}:=\mathrm{Ann}\,\V \subset \T^{\ast}M$\, its annihilator. By an Ehresmann connection on the fiber bundle, we mean a \emph{horizontal subbundle} \,$\HH \subset \T{M}$,\, that is, a complementary subbundle to the vertical one,
     \begin{equation}
          \T{M} \,=\, \HH \oplus \V, \label{D1}
     \end{equation}
This induces the dual decomposition
     \begin{equation}
          \T^{\ast}M \,=\, \V^{\circ} \oplus \HH^{\circ}. \label{D2}
     \end{equation}

Alternatively, one can define an Ehresmann connection as a vector-valued $1$-form \,$\gamma \in \Omega^{1}(M;\V)$\, (i.e. a vector bundle morphism \,$\gamma:\T{M} \rightarrow \V$)\, with property \,$\gamma|_{\V}=\mathrm{id}_{\V}$.\, Then, given a horizontal subbundle $\HH$, the connection form is define by the natural projection \,$\gamma=\mathrm{pr}_{2}:\T{M} \longrightarrow \V$\, along $\HH$. On the contrary, given a $\gamma$ we put \,$\HH:=\ker\gamma$.

Given an Ehresmann connection $\gamma$, a horizontal lift of a vector field \,$u \in \X{B}$\, is the unique horizontal vector field \,$\hor^{\gamma}u \in \Gamma\HH$\, which is $\pi$-related to $u$. Moreover, the horizontal lift of every $k$-vector field $\psi$ on $B$ is defined as a horizontal $k$-vector field \,$\hor^{\gamma}\psi \in \Gamma\wedge^{k}\HH$\, satisfying the condition $\hor^{\gamma}\psi\,(\pi^{\ast}\alpha_{1},\ldots,\pi^{\ast}\alpha_{k}) = \pi^{\ast}(\psi(\alpha_{1},\ldots,\alpha_{k}))$,\, for \,$\alpha_{1},\ldots,\alpha_{k} \in \Gamma\,\T^{\ast}{B}$.\, The curvature of the connection $\gamma$ is a vector valued $2$-form \,$\Curv^{\gamma}\in\Omega^{2}(B;\V)$\, on the base $B$ given by
     \begin{equation*}
          \Curv^{\gamma}(u_{1},u_{2}) \,:=\, [\hor^{\gamma}u_{1},\hor^{\gamma}u_{2}]-\hor^{\gamma}[u_{1},u_{2}] \,\in\, \Gamma\V,
     \end{equation*}
for \,$u_{1},u_{2} \in \X{B}$.\, As is known that the horizontal subbundle $\HH$ is integrable if and only if the connection $\gamma$ is flat, \,$\Curv^{\gamma}=0$.

Fix a local coordinate system $(x^{i},y^{a})$ on the total space $M$, where \,$x=(x^{i})$\, are coordinates on the base and \,$y=(y^{a})$\, are coordinates along the fiber of $\pi$. Then, \,$\gamma=(\dd{y^{a}}+\gamma_{i}^{a}\,\dd{x^{i}}) \otimes \frac{\partial}{\partial y^{a}}$\, and \,$\hor_{i}^{\gamma}:=\hor^{\gamma}\left(\frac{\partial}{\partial x^{i}}\right)=\frac
{\partial}{\partial x^{i}}-\gamma_{i}^{a}\,\frac{\partial}{\partial y^{a}}$,\, the summation on repeated indices will be understood. So, we have the following local descriptions of the horizontal and vertical distributions
     \begin{equation*}
          \HH \,=\, \mathrm{span}\Big\{\hor_{1}^{\gamma},\ldots,\hor_{n}^{\gamma}\Big\}, \qquad \V \,=\, \mathrm{span}\left\{ \tfrac{\partial}{\partial y^{1}},\ldots,\tfrac{\partial}{\partial y^{r}} \right\}.
     \end{equation*}
Here \,$n=\dim{B}$\, and \,$r=\rank{\pi}$.\, Taking the dual basis $\{\dd{x^{i}},\eta^{a}\}$ of $\{\hor_{i}^{\gamma},\frac{\partial}{\partial y^{a}}\}$, where \,$\eta^{a}:=\gamma_{i}^{a}\,\dd{x^{i}}+\dd{y^{a}}$,\, we also
get
     \begin{equation*}
          \V^{\circ} \,=\, \mathrm{span}\big\{\dd{x^{1}},\ldots,\dd{x^{n}}\big\}, \qquad \HH^{\circ} \,=\, \mathrm{span}\big\{\eta^{1},\ldots,\eta^{r}\big\}.
     \end{equation*}
Decompositions (\ref{D1}) and (\ref{D2}) give the bigrading for multivector fields and differential forms on $M$. We say that a tensor field on $M$ is of bidegree $(p,q)$ if this field is locally generated by the elements of the form \,$\hor_{i_{1}}^{\gamma} \wedge\cdots\wedge \hor_{i_{p}}^{\gamma} \wedge \tfrac{\partial}{\partial y^{a_{1}}} \wedge\cdots\wedge \tfrac{\partial}{\partial y^{a_{q}}}$.\, Therefore, the indices $p$ and $q$ denote the degrees in the direction of $\HH$ and $\V$, respectively. Each tensor field $A$ on $M$ has a $\gamma$-dependent bigraded decomposition whose component of bidegree $(p,q)$ will be denoted by $A_{p,q}$. The same bigrading argument are applied to differential forms on $E$, in particular, a form of bidegree $(p,q)$ is a linear combination of basic elements \,$\dd{x^{i_{1}}} \wedge\cdots\wedge \dd{x^{i_{p}}} \wedge \eta^{a_{1}} \wedge\cdots\wedge \eta^{a_{q}}$.\, Moreover, we say that a linear operator, acting on the space of tensor fields or differential forms, has a bidegree $(t,s)$ if it sends $(p,q)$-elements to elements of bidegree $(p+t,q+s)$. For example, the exterior differential $\dd$ for forms on $M$ has the following bigraded decomposition \cite{Va-04}: \,$\dd=\dd_{1,0}+\dd_{2,-1}+\dd_{0,1}$.\, Here \,$\dd_{p,q}=\dd_{p,q}^{\gamma}$\, is a $\gamma$-dependent operator of bidegree $(p,q)$. The coboundary condition for $\dd$ imply the relations: \,$\dd_{1,0}^{2} + \dd_{2,-1} \circ \dd_{0,1}  + \dd_{0,1} \circ \dd_{2,-1} = 0$, \,$\dd_{1,0} \circ \dd_{0,1} + \dd_{0,1} \circ \dd_{1,0} =0$\, and \,$\dd_{0,1}^{2} =0$.

It is also useful to note that two Ehresmann connections $\gamma$ and $\widetilde{\gamma}$ on $M$ are related by
     \begin{equation}
          \widetilde{\gamma} \,=\, \gamma-\Xi,\label{CT1}%
     \end{equation}
where a vector valued $1$-form \,$\Xi \in \Omega^{1}(M;\T{M})$\, satisfies the conditions
     \begin{equation}
          \mathrm{Im}\,\Xi \,\subseteq\, \V \,\subseteq\, \ker\Xi. \label{CT2}
     \end{equation}
It follows that the horizontal subbundle $\mathbb{\widetilde{H}}$ associated with $\widetilde{\gamma}$ is given by \,$\widetilde{\HH}=(\mathrm{id}+\Xi)(\HH)$.

\section{Almost Coupling Poisson Tensors}

Let \,$\pi:M \rightarrow B$\, be a fiber bundle. Given a bivector field \,$\Pi \in \Gamma\wedge^{2}\T{M}$,\, let us pick an arbitrary Ehresmann connection $\gamma$ on $M$ and consider the corresponding bigraded decomposition
     \begin{equation}
          \Pi \,=\, \Pi_{2,0} + \Pi_{1,1} + \Pi_{0,2} \label{DF}
     \end{equation}
Here \,$\Pi_{2,0} \in \Gamma\wedge^{2}\HH$\, and \,$\Pi_{0,2} \in \Gamma\wedge^{2}\V$\, are the horizontal and vertical components, respectively. It is clear that
     \begin{equation*}
          \rank{\Pi_{2,0}} \,\leq\, \rank{\HH} \,=\, \dim{B}.
     \end{equation*}
Under varying the connection $\gamma$, the components in decomposition (\ref{DF}) are changing by some rules according to (\ref{CT1}). But it is easy to see that the rank of the horizontal part regarded as a function \,$\rank{\Pi_{2,0}}:M \rightarrow \mathbb{Z}$\, does not depend on the choice of $\gamma$. As a consequence, the open subset
     \begin{equation}
          M^{\Pi} \,:=\, \{\, p \in M \,\big|\, \rank_{p}\Pi_{2,0} = \dim{B} \,\} \label{CO}
     \end{equation}
in $M$ consisting of all point \,$p \in M$\, at which the rank of the horizontal bivector field is maximal, is also $\gamma$-independent and hence represents an intrinsic characteristic of $\Pi$. Here we suppose that the dimension of $B$ is even.

We are interested in the following class of compatible bivector fields on the fibered manifold $M$ \cite{YuV-2001,Va-04}.

\begin{definition}
A bivector field \,$\Pi \in \Gamma\wedge^{2}\T{M}$\, is said to be an \emph{almost coupling tensor} if there exists an Ehresmann connection $\gamma$ on $M$ such that the mixed term in the $\gamma$-dependent decomposition (\ref{DF})
vanishes,
     \begin{equation}
          \Pi_{1,1} \,=\, 0, \label{AC}
     \end{equation}
or, explicitly, \,$\Pi(\alpha,\beta)=0$,\, for all \,$\alpha \in \V^{\circ}$\, and \,$\beta \in \HH^{\circ}$.
\end{definition}

Equivalently, condition (\ref{AC}) can be reformulated as follows, there exists a horizontal subbundle \,$\HH \subset \T{M}$\, such that
     \begin{equation}
          \Pi^{\sharp}\big(\V^{\circ}\big) \,\subseteq\, \HH. \label{AC1}
     \end{equation}
Here \,$\Pi^{\sharp}:\T^{\ast}M \rightarrow \T{M}$\, is a vector bundle morphism given by \,$\alpha \mapsto \mathbf{i}_{\alpha}\Pi$.\, In particular, $\Pi$ is called a \emph{coupling bivector field} if
     \begin{equation}
          \HH \,=\, \Pi^{\sharp}\big(\V^{\circ}\big) \label{AC2}
     \end{equation}
is a horizontal subbundle. Therefore, in the coupling case, there exists a \emph{unique} Ehresmann connection $\gamma$ on $M$ defined by (\ref{AC2}) which provides the property (\ref{AC1}) for the bivector field $\Pi$.

Recall that \cite{YuV-2001,Va-04} a coupling tensor $\Pi$ on $M$ is uniquely determined by the so-called \emph{geometric data} $(\gamma,\sigma,\Pi_{0,2})$, consisting of the Ehresmann connection $\gamma$, a horizontal $2$-form \,$\sigma \in \Gamma\wedge^{2}\V^{\circ}$,\, called the \emph{coupling form} and the vertical component $\Pi_{0,2}$ of $\Pi$. By the horizontal nondegeneracy of $\sigma$, the horizontal component of $\Pi$ is recovered by the formula \,$\mathbf{i}_{\mathbf{i}_{\alpha}\Pi_{2,0}}\sigma = -\alpha$,\, for all \,$\alpha \in \V^{\circ}$.

Note that the (almost) coupling condition is natural with respect to the restriction of \ bivector fields to open subsets in $M$.

\begin{lemma}
For a bivector field \,$\Pi \in \Gamma\wedge^{2}\T{M}$,\, its restriction to the open subset \,$M^{\Pi} \neq \O$\, is a coupling tensor.
\end{lemma}
\begin{proof}
Let us show that for each \,$p \in M^{\Pi}$,\, the bivector field $\Pi$ satisfies the coupling property which splits into the following conditions:
     \begin{equation}
          \Pi^{\sharp}\big(\V_{p}^{\circ}\big) \cap \V_{p} \,=\, \{0\} \qquad \text{and} \qquad \dim\Pi^{\sharp}\big(\V_{p}^{\circ}\big) \,=\, \dim{B} \label{SV1}.
     \end{equation}
Starting with decomposition (\ref{DF}) with respect to a fixed connection $\gamma$, we define a ``new'' Ehresmann connection $\widetilde{\gamma}$ on $M^{\Pi}$ by formula (\ref{CT1}), where a vector valued $1$-form $\Xi$ is given by the relation \,$\Pi_{1,1}^{\sharp}|_{\V_{p}^{\circ}} = \big(\Xi \circ \Pi_{2,0}^{\sharp} \big) |_{\V_{p}^{\circ}}$.\, By the definition of $M^{\Pi}$, the restriction $\Pi_{2,0}^{\sharp}|_{\V_{p}^{\circ}}$ is invertible for all \,$p \in M^{\Pi}$.\, Then, the fact that $\Pi|_{M^{\Pi}}$ is an almost coupling tensor via $\widetilde{\gamma}$ together with (\ref{CO}) imply the relations (\ref{SV1}).
\end{proof}

\begin{definition}
The open subset $M^{\Pi}$ in the fibered manifold $M$ will be called a \emph{coupling domain} of a bivector field \,$\Pi \in \Gamma\wedge^{2}\T{M}$.
\end{definition}

Now, given an almost coupling tensor $\Pi$ via an Ehresmann connection $\gamma$ on the fibered manifold $M$, we have the decomposition
     \begin{equation*}
          M \,=\, M^{\Pi} \,\cup\, \partial M^{\Pi} \,\cup\, \mathrm{Int}\big( M \setminus M^{\Pi} \big). \label{Dec}
     \end{equation*}

We observe that the Ehresmann connection $\gamma$ is uniquely defined by $\Pi$ in the closure $\overline{M^{\Pi}}$ of the coupling domain which coincides with the whole $M$ in the case when $M^{\Pi}$ is dense in $M$. But, in general, $\gamma$ is not fixed in the interior of the complement $M \setminus M^{\Pi}$. Note also that each point $p$ of the boundary $\partial M^{\Pi}$ is a \emph{singular point} of $\Pi$ in the sense that $\rank{\Pi}$ is not locally constant around $p$. Clearly, $\Pi$ is a coupling tensor on $M$ if and only if \,$M^{\Pi}=M$.

The next point is to study the set of\emph{ almost coupling Poisson tensors}, that is, the almost coupling bivector fields \,$\Pi=\Pi_{2,0}+\Pi_{0,2}$\, satisfying the Jacobi identity \,$\cSch{\Pi,\Pi}=0$,\, where $\cSch{,}$ denotes the Schouten bracket for multivector fields on $M$ \cite{Va-94}. Geometrically, the Jacobi identity means that the characteristic distribution
     \begin{equation}
          C^{\Pi} \,:=\, \Pi^{\sharp}\big( \T^{\ast}M \big) \,=\, \Pi_{2,0}^{\sharp}\big(\V^{\circ}\big) \oplus \Pi_{0,2}^{\sharp}\big(\HH^{\circ}\big), \label{CD}
     \end{equation}
is integrable in the sense of Sussman-Stefan and gives rise to the symplectic foliation $(\mathcal{S},\varpi)$.

By using the bigrading arguments and the properties of the Schouten bracket, one can show that the Jacobi identity for $\Pi$ splits into the following equations for the horizontal \,$\Pi_{2,0}\in\Gamma\wedge^{2}\HH$\, and vertical \,$\Pi_{0,2}\in\Gamma\wedge^{2}\V$\, components:
     \begin{align}
          \cSch{\Pi_{2,0},\Pi_{2,0}}_{3,0} \,&=\, 0, \label{J1} \\[0.15cm]
          \cSch{\Pi_{2,0},\Pi_{2,0}}_{2,1} \,+\, 2\cSch{\Pi_{2,0},\Pi_{0,2}}_{2,1} \,&=\, 0, \label{J2} \\[0.15cm]
          \cSch{\Pi_{2,0},\Pi_{0,2}}_{1,2} \,&=\, 0, \label{J3} \\[0.15cm]
          \cSch{\Pi_{0,2},\Pi_{0,2}} \hspace{0.41cm} \,&=\, 0. \label{J4}
     \end{align}
In particular, the last equality is just the Jacobi identity for the vertical component which says that $\Pi_{0,2}$ is a Poisson bivector field. It follows from (\ref{CD}) that $C^{\Pi}$ is the sum of the characteristic distribution of the vertical Poisson tensor $\Pi_{0,2}$ and the horizontal factor \,$\Pi_{2,0}^{^{\sharp}}(\V^{\circ})\subseteq\HH$\, which is not necessarily integrable, in general. This happens in the flat case \cite{YuV-2001}.

\begin{proposition}\label{PropFlatPair}
In the coupling domain $M^{\Pi}$ of an almost coupling Poisson tensor $\Pi$, the following conditions are equivalent:
     \begin{itemize}
       \item[$(i)$] the curvature of $\gamma$ is zero, \,$\Curv^{\gamma}=0$;
       \item[$(ii)$] the horizontal bivector field $\Pi_{2,0}$ is a Poisson tensor;
       \item[$(iii)$] the horizontal $\Pi_{2,0}$ and vertical $\Pi_{0,2}$ components of $\Pi$ form a Poisson pair.
     \end{itemize}
\end{proposition}

We have the following consequence of Proposition \ref{PropFlatPair}. Let $M$ be a $3$-dimensional fibered manifold over $2$-manifold $B$ and $\Pi$ an almost coupling Poisson tensor on $M$ via a connection $\gamma$. Then, $\Pi$ is a horizontal bivector field and hence in the coupling domain \,$M^{\Pi}=\{p \in M \,|\, \rank_{p}\Pi =2\}$,\, the connection $\gamma$ is necessarily flat. %$\Pi$ is the $\gamma$-horizontal lift of a bivector field $\psi$ on $B$ and \,$\overline{M^{\Pi}}=\pi^{-1}(\mathrm{supp}\psi)$.

The symplectic foliation of an almost coupling Poisson tensor $\Pi$ is described as follows. Let $(S,\varpi)$ and $(L,\tau)$ be two symplectic leaves through a point \,$p \in M^{\Pi}$\, associated to the Poisson structures $\Pi$ and $\Pi_{0,2}$, respectively. Then, \,$\T_{p}S=\HH_{p} \oplus \T_{p}L$\, and for symplectic form we have \,$\varpi_{p}=\sigma_{p}\oplus\tau_{p}$,\, where $\sigma$ is the coupling form of $\Pi|_{M^{\Pi}}$. If $p$ approaches the boundary $\partial M^{\Pi}$, then the first term in this sum becomes singular. Moreover, one can show that \,$L \subset S \subset M^{\Pi}$\, and hence the coupling domain is invariant relative to the Hamiltonian flows. This follows from the property that the boundary $\partial{M^{\Pi}}$ consist of singular points of $\Pi$.

Recall that a immerse submanifold \,$N \subset M$\, is said to be a Poisson submanifold of $(M,\Pi)$ if \,$C^{\Pi}|_{N} \subseteq \T{N}$.\, In this case, there exists a unique Poisson structure on $N$ such that the inclusion \,$N \hookrightarrow M$\, is a Poisson map. Its clear that each symplectic leaf of $\Pi$ is a Poisson submanifold.

By using (\ref{CD}) we derive the following fact.

\begin{proposition}\label{PropSubPoisson}
Let $\Pi$ be an almost coupling Poisson tensor on $M$ via a horizontal subbundle $\mathbb{H}$ and \,$s:B \rightarrow M$\, a smooth section of $\pi$. Then, \,$s(B) \subset M$\, is a Poisson submanifold of $(M,\Pi)$ if and only if
	\begin{align}
		\Pi_{2,0}^{\sharp}\big( \V^{\circ}|_{s(B)} \big) \,&\subseteq\, \T{\big(s(B)\big)}, \label{PoiSub1} \\[0.15cm]
		\Pi_{0,2}^{\sharp}|_{s(B)} \,&=\, 0. \nonumber %\label{PoiSub2}
	\end{align}
\end{proposition}

So, condition (\ref{PoiSub1}) means that $\Pi_{2,0}$ is tangent to $s(B)$ and hence the restriction $\Pi_{2,0}|_{s(B)}$ gives the Poisson structure on $s(B)$. It is clear that condition (\ref{PoiSub1}) holds if \,$\HH|_{s(B)} \subseteq \T{\big(s(B)\big)}$.
	
%\begin{remark}
%Under conditions (3.13) and (3.14), one can try to apply to $\Pi$ the linearization procedure analogous to \cite{YuV-2001}, to get  a linearized almost coupling Poisson structure around the Poisson submanifold $s(B)$. Here we do not consider this problem in a general setting and it will be discussed elsewhere.
%\end{remark}

For an open subset \,$U\subseteq M$,\, the space of Casimir functions of a Poisson tensor $\Pi$ restricted to $U$ is denoted by \,$\mathrm{Casim}(U,\Pi) \,:=\, \{ h \in \Cinf{U} \,\big|\, \mathbf{i}_{\dd{h}}\Pi=0\}$.\, The almost coupling property of $\Pi$ implies the following relation between the Casimir functions of $\Pi$ and its vertical part $\Pi_{0,2}$,
     \begin{equation*}
          \mathrm{Casim}(U,\Pi) \,\subseteq\, \mathrm{Casim}(U,\Pi_{0,2}).
     \end{equation*}
Note also that \,$\pi^{\ast}\Cinf{B} \subseteq \mathrm{Casim}(M,\Pi_{0,2})$.

We conclude this section with some remarks on the natural symmetry group of transformations which leave invariant the set of all almost coupling Poisson structures on a given fiber bundle \,$\pi:M\rightarrow B$.\, Let \,$g:M \rightarrow M$\, be a fiber preserving diffeomorphism on the total space, that is, $g$ descends to a diffeomorphism \,$g_{0}:B \rightarrow B$\, on the base, \,$\pi \circ g = g_{0} \circ \pi$.\, The key property is that, the tangent map \,$\dd{g}:\T{M} \rightarrow \T{M}$\, leaves invariant the vertical subbundle $\V$, \,$(\dd_{p}g)(\V_{p})=\V_{g(p)}$.\, As a consequence, the pull-back by $g$ of a Ehresmann connection $\gamma$ on $M$ is well defined as \,$(g^{\ast}\gamma)(X)=g^{\ast}(\gamma(g_{\ast}X))$,\, for \,$X \in \X{M}$.\, Then, we have the following fact \cite{YuV-2001,Va-04}; let \,$\Pi=\Pi_{2,0}+\Pi_{0,2}$\, be an almost coupling Poisson structure via a connection $\gamma$ and \,$g:M \rightarrow M$\, be a fiber preserving diffeomorphism. Then, the pull-back $g^{\ast}\Pi$ is again an almost coupling Poisson structure via the connection $g^{\ast}\gamma$, which is the sum of the bigraded components \,$\big(g^{\ast}\Pi\big)_{2,0}=g^{\ast}\Pi_{2,0}$\, and \,$\big(g^{\ast}\Pi\big)_{0,2}=g^{\ast}\Pi_{0,2}$.

     \section{The Case of $(2+3)$-Fibered Manifolds}

In this section, we focus on the detailed study of the further properties of almost coupling Poisson structure in the case of $5$-dimensional fibered manifolds with $2$-dimensional bases.

Let $M$ be a $5$-dimensional orientable manifold equipped with a volume form $\Omega$. Assume that $M$ is a fibered manifold in the sense that we are given a fiber bundle \,$\pi:M \rightarrow B$\, over an oriented $2$-manifold $B$ carrying an area (symplectic) form $\omega$. Therefore, the fibers of $\pi$ are $3$-dimensional orientable manifolds and \,$\rank{\V}=3$.

Suppose also that we are given a triple $(\gamma,\kappa,\beta)$ consisting of an Ehresmann connection \,$\gamma\in\Omega^{1}(E;\V)$,\, a smooth function \,$\kappa \in \Cinf{M}$\, and a vertical $1$-form \,$\beta \in \Gamma\HH^{\circ}$.\, Here \,$\HH=\ker{\gamma} \subset \T{M}$\, is the horizontal subbundle of rank $2$.

Denote the pull-back of the area form $\omega$ to $M$ by
     \begin{equation*}\label{OmegaHdef}
          \OmegaH \,=\, \pi^{\ast}\omega \,\in\, \Gamma\wedge^{2}\V^{\circ}.
     \end{equation*}
Then, there exists a unique nowhere vanishing $3$-form \,$\OmegaV \in \Gamma\wedge^{3}\HH^{\circ}$\, of bidegree $(0,3)$ such that
	\begin{equation}\label{OmegaDef}
		\Omega \,=\, \OmegaH \wedge \OmegaV.
	\end{equation}
Moreover, there exist nowhere vanishing ``dual'' tensor fields \,$\QH\in \Gamma\wedge^{2}\HH$\, and \,$\QV \in \Gamma\wedge^{3}\V$\, such that
     \begin{equation}
          \mathbf{i}_{\QH}\OmegaH \,=\, 1 \qquad \text{and} \qquad \mathbf{i}_{\QV}\OmegaV \,=\, 1. \label{PP}
     \end{equation}
Here the interior product of multivectors fields and forms is defined by the rule \,$\mathbf{i}_{X \wedge Y}=\mathbf{i}_{X} \circ \mathbf{i}_{Y}$,\, for any \,$X,Y \in \X{M}$.\, We emphasize that the elements $\OmegaV$ and $\QH$ are $\gamma$-dependent. Moreover, it follows from (\ref{PP}) that $\QH$ is the $\gamma$-horizontal lift of a bivector field on the base,
     \begin{equation}
          \QH \,=\, -\hor^{\gamma}\psi, \label{PP1}
     \end{equation}
where \,$\psi \in \Gamma\wedge^{2}\T{B}$\, is the nondegenerate Poisson tensor of the symplectic $2$-manifold $(B,\omega)$, \,$\mathbf{i}_{\psi}\omega=-1$.

Now, let us associate to the triple $(\gamma,\kappa,\beta)$ the following $2$-tensor field on $E$:
     \begin{equation}
          \Pi \,=\, \kappa\,\hor^{\gamma}\psi + P_{\beta} \label{EP1}
     \end{equation}
Here, the vertical bivector field
     \begin{equation}
          P_{\beta} \,:=\, -\mathbf{i}_{\beta}\QV \,\in\, \Gamma\wedge^{2}\V, \label{EP2}
     \end{equation}
is independent of $\gamma$ and uniquely determined by the restriction (a fibered $1$-form) \,$\overline{\beta}=\beta|_{\V} \in \Gamma\V^{\ast}$.\, Its clear that $\Pi$ is an almost coupling bivector field with components \,$\Pi_{2,0}=-\kappa\,\QH$\, and \,$\Pi_{0,2}=P_{\beta}$\, in the $\gamma$-bigraded decomposition. Introduce also the following $\gamma$-dependent forms
     \begin{equation}
          \theta \,:=\, -\mathbf{i}_{\QV}\dd_{1,0}\OmegaV \,\in\, \Gamma\V^{\circ} \label{F1}
     \end{equation}
and
     \begin{equation}
          \hspace{0.5cm} \varrho \,:=\, \mathbf{i}_{\QH}\dd_{2,-1}\OmegaV \,\in\, \Gamma\wedge^{2}\HH^{\circ}. \label{F2}
     \end{equation}
One can show that these forms are solutions to the equations
     \begin{equation}
          \dd_{1,0}\OmegaV \,=\, \theta \wedge \OmegaV \qquad \text{and} \qquad \dd_{2,-1}\OmegaV \,=\, \OmegaH \wedge \varrho. \label{F4}
     \end{equation}
We arrive at the following result.

\begin{theorem}\label{TeoJacobiOrientable}
The almost coupling bivector field $\Pi$ in (\ref{EP1}) is a Poisson tensor if and only if the triple $(\gamma,\kappa,\beta)$ satisfies the \ following integrability conditions:
     \begin{align}
          \dd_{0,1}\beta \wedge \beta \,&=\, 0, \label{IC1} \\[0.15cm]
          \kappa\,\left( \dd_{1,0}\beta \,-\, \theta \wedge \beta \right) \,&=\, 0, \label{IC2} \\[0.15cm]
          \dd_{0,1}\kappa \wedge \beta \,-\, \kappa^{2}\varrho \,&=\, 0. \label{IC3}
     \end{align}
\end{theorem}

The proof of this theorem is a direct verification of the fact that relations (\ref{J1})-(\ref{J4}), representing the bigraded factorization of the Jacobi identity for $\Pi$ (\ref{EP1}), are equivalent with equations (\ref{IC1})-(\ref{IC3}). Here we use some properties of the Schouten bracket and identities (\ref{F4}). Notice that the relation (\ref{J1}) holds automatically because of the dimension argument.

In particular, equation (\ref{IC1}) for $\beta$ represents the Jacobi identity for the vertical bivector field $P_{\beta}$ (\ref{EP2}) which induces the Poisson fiber bundle $(M \overset{\pi}{\rightarrow} B,P_{\beta})$. The corresponding fiberwise Poisson structure is obtained under the restriction of $P_{\beta}$ to the $3$-dimensional fibers of $\pi$. The restriction \,$\overline{\beta}=\beta|_{\V}\in\Gamma\V^{\ast}$\, is said to be a fibered Poisson $1$-form and satisfies the condition \,$\dd_{\mathcal{V}}\overline{\beta} \wedge \bar{\beta}=0$\, involving the foliated de Rham differential $\dd_{\mathcal{V}}$ along the fibers of $\pi$. Notice that the equation (\ref{IC1}) is invariant under the transformation \,$\beta \mapsto f\beta$,\, for arbitrary \,$f \in \Cinf{M}$.\, This is a fiber bundle version of the well-known property of the conformal invariance of the Jacobi identity in the $3$-dimensional case \cite{GN-93}.

A setup $(\gamma,\kappa,\beta)$ satisfying equations (\ref{IC1})-(\ref{IC3}) will be called a \emph{Poisson triple} relative to the pair $(\Omega,\omega)$.

\begin{remark}
Under changing the volume form \,$\widetilde{\Omega}=m\Omega$,\, where $m$ is a positive smooth function on $M$, a Poisson triple $(\gamma,\kappa,\beta)$ relative to $(\Omega,\omega)$ is transformed to the Poisson triple \,$\big( \widetilde{\gamma}=\gamma,\, \widetilde{\beta}=m\beta,\, \widetilde{\kappa}=\kappa \big)$\, relative to $(\widetilde{\Omega},\omega)$, which is a solution to equations (\ref{IC1})-(\ref{IC3}) with \,$\widetilde{\theta}=\theta + \frac{1}{m}\dd_{1,0}m$\, and \,$\widetilde{\varrho}=m\varrho$. \hfill{$\triangleright$}
\end{remark}

We have also the converse to Theorem \ref{TeoJacobiOrientable} which says that each almost coupling Poisson structure $\Pi$ on the fibered manifold $5$-manifold $M$ is represented by (\ref{EP1}).

\begin{theorem}\label{TeoACouptoTriple}
Every almost coupling bivector field \,$\Pi=\Pi_{2,0}+\Pi_{0,2}$\, via a horizontal subbundle $\HH$ is of the form (\ref{EP1}), where the Poisson triple $(\gamma,\kappa,\beta)$ is given by
     \begin{equation*}
          \gamma \,=\, \mathrm{pr}_{2}:\T{M} \overset{\HH}{\rightarrow} \V, \qquad \kappa \,=\, -\mathbf{i}_{\Pi_{2,0}}\pi^{\ast}\omega, \qquad \beta \,=\, -\mathbf{i} _{\Pi_{0,2}}\OmegaV.
     \end{equation*}
\end{theorem}

Therefore, the assignment \,$(\gamma,\kappa,\beta) \mapsto \Pi$\, is surjective. But it is not one-to-one.

\begin{example}
Let $\gamma$ be an arbitrary connection on $M$ and \,$\beta \in \Gamma\HH^{\circ}$\, a vertical $1$-form satisfying the fiberwise Jacobi identity. Then, the triple $(\gamma,0,\beta)$ is a solution to equations (\ref{IC1})-(\ref{IC3}) which gives the vertical Poisson tensor \,$\Pi=P_{\beta}.$ \hfill{$\triangleleft$}
\end{example}

Now, let us describe some properties of the almost coupling Poisson tensor \,$\Pi=\Pi_{2,0}+\Pi_{0,2}$\, (\ref{EP1}) associated to a given Poisson triple $(\gamma,\kappa,\beta)$. In terms of the setup $(\gamma,\kappa,\beta)$, the Poisson bracket defined by $\Pi$ can be written as follows
     \begin{equation*}
          \{f,g\} \,=\, \kappa\,\frac{\dd_{1,0}f \wedge \dd_{1,0}g}{\pi^{\ast}\omega} \,+\, \frac{\dd_{0,1}f \wedge \dd_{0,1}g \wedge \beta}{\OmegaV}.
     \end{equation*}
The second term on the right hand side of this equality represents the Poisson bracket of \,$\Pi_{0,2}=P_{\beta}$\, which will be denoted by $\{,\}_{\beta}$. The Hamiltonian vector field \,$X_{F}=\mathbf{i}_{\dd{F}}\Pi$\, of a function \,$F \in \Cinf{M}$\, has the bigraded decomposition: \,$X_{F}=(X_{F})_{1,0}+(X_{F})_{0,1}$,\, where
     \begin{equation}
          \big(X_{F}\big)_{1,0} \,=\, \kappa\,\mathbf{i}_{\dd_{1,0}F}\hor^{\gamma}\psi, \qquad \big(X_{F}\big)_{0,1} \,=\, \mathbf{i}_{\dd{F}}P_{\beta} \,=\, -\mathbf{i}_{\dd_{0,1}F \wedge \beta}\QV. \label{PVF1}
     \end{equation}
It follows that the characteristic distribution $C^{\Pi}$ of $\Pi$ is generated by vector fields \,$\kappa\,\hor^{\gamma}u$\, and \,$\mathbf{i}_{\dd{F} \wedge \beta}\QV$,\, for any \,$u \in \X{B}$,\, $F \in \Cinf{M}$.\, The space of Casimir functions $\mathrm{Casim}(M,\Pi)$ consists of all \,$c \in \Cinf{M}$\, such that
	\begin{equation}\label{CasimACoup}
		\kappa\,\dd_{1,0}c \,=\, 0 \qquad \text{and} \qquad \dd_{0,1}c \wedge \beta \,=\, 0.
	\end{equation}
The second condition just determines the space of Casimir functions of the vertical Poisson structure $P_{\beta}$.

Denote by $\mathcal{Z}(\kappa)$ and $\mathcal{Z}(\beta)$ the zero sets of the function $\kappa$ and the $1$-form $\beta$, respectively. It is clear that $\mathcal{Z}(\beta)$ is also the zero set of $P_{\beta}$. Moreover, the coupling domain of $\Pi$ is just \,$M^{\Pi}=M\setminus\mathcal{Z}(\kappa)$\, and we have the decomposition:
     \begin{equation}
          M \,=\, M^{\Pi} \,\cup\, \partial\big(\mathcal{Z}(\kappa)\big) \,\cup\, \mathrm{Int}\big(\mathcal{Z}(\kappa)\big). \label{D}
     \end{equation}
So, the rank of $\Pi$ is zero in $\mathcal{Z}(\kappa)\cap\mathcal{Z}(\beta)$ and
     \begin{align*}
          \{\rank{\Pi}=2\} \,&=\, M^{\Pi} \,\cap\, \mathcal{Z}(\beta) \ \,\scalebox{0.9}{$\bigcup$}\, \ \mathcal{Z}(\kappa) \,\cap\, \big( M \setminus \mathcal{Z}(\beta) \big), \\[0.15cm]
          \{\rank{\Pi=4}\} \,&=\, M^{\Pi} \,\cap\, \big(M\setminus\mathcal{Z}(\beta)\big).
     \end{align*}
Therefore, the characteristic foliation $\mathcal{S}$ of $\Pi$ consists of symplectic leaves of dimension $0,2,4$. In particular, there are two kinds of $2$-dimensional symplectic leaves. If \,$\kappa(p)=0$\, and \,$\beta_{p} \neq 0$,\, then the symplectic leaf $S_{p}$ of $\Pi$ through \,$p \in M$\, coincides with the symplectic leaf of the restricted vertical Poisson structure $P_{\beta}|_{M_{b}}$ on the $3$-dimensional fiber $M_{b}$ of $\pi$ over \,$b=\pi(p)$.\, If \,$\kappa(p) \neq 0$\, and \,$\beta_{p}=0$,\, then by (\ref{IC3}) at each point \,$q \in S_{p}$,\, the curvature of the connection $\gamma$ vanishes and the tangent space to $S_{p}$ coincides with the horizontal plane  $\mathbb{H}_{q}$. The corresponding symplectic form is the pull-back of \,${\pi^{\ast}\omega}/{\kappa}$\, to $S_{p}$. Notice also that the coupling domain $M^{\Pi}$ is foliated by the symplectic leaves of dimension $2$ and $4$.

\begin{remark}
If \,$M^{\Pi} \cap \mathrm{supp}{\,\beta}=\O$,\, then the rank of $\Pi$ equals $0$ or $2$ and the Poisson structure $\Pi$ is of the Flaschka-Ratiu type \cite{Damianou}. Such class appears in the problem of the construction of Poisson structures with prescribed characteristic foliations \cite{Damianou,Nar-2015,PabSua-2018}. \hfill{$\triangleright$}
\end{remark}

Now, we observe that, according to decomposition (\ref{D}), equations (\ref{IC2}), (\ref{IC3}) for $(\gamma,\kappa,\beta)$ can be represented as follows.
     \begin{itemize}
          \item In the open subset $M^{\Pi}$:
               \begin{align}
                    \dd_{1,0}\beta \,+\, \beta \wedge \theta \,&=\, 0, \label{C2} \\[0.15cm]
                    \dd_{0,1}\left(\tfrac{1}{\kappa}\right) \wedge \beta \,+\, \varrho \,&=\, 0; \label{C3}
               \end{align}

          \item in the closed subset $\mathcal{Z}(\kappa)$:\quad $\dd_{0,1}\kappa \wedge \beta=0$.
     \end{itemize}
In fact, the last condition holds automatically in the interior $\mathrm{Int}\big(\mathcal{Z}(\kappa)\big)$ and hence, can be rewritten in the form:
     \begin{equation}
          (\dd_{0,1}\kappa)_{p} \wedge \beta_{p} \,=\, 0, \quad \forall\, p \in \partial\big(\mathcal{Z}(\kappa)\big). \label{C5}
     \end{equation}
From here and the relation \,$\partial\big(\mathcal{Z}(\kappa)\big)=\partial M^{\Pi}$,\, we conclude that condition (\ref{C5}) involves only the first variation of $\kappa$ at the boundary points of the coupling domain.

Moreover, relations (\ref{C2}) and (\ref{C3}) have the following interpretations. Using (\ref{F1}) and taking the interior product of both sides of equation (\ref{C2}) with $\QH  \wedge \QV$, we show that (\ref{C2}) is equivalent to the condition:
     \begin{equation}
          \dlie{\hor^{\gamma}u}P_{\beta} \,=\, 0, \quad \forall\, u \in \X{B}, \label{PA}
     \end{equation}
which means that the $\gamma$-horizontal lift $\hor^{\gamma}u$ is an \emph{infinitesimal automorphism} of the vertical Poisson bivector field $P_{\beta}$ restricted to $M^{\Pi}$. In other words, $\gamma$ is a \emph{Poisson connection} on the corresponding Poisson fiber bundle $\pi|_{M^{\Pi}}$.

We have the following consequence of properties (\ref{PA}) and (\ref{PP1}): the bivector field \,$\QH \in \Gamma\wedge^{2}\HH$\, given by (\ref{PP}) is a $2$-cocycle in the Lichnerowicz-Poisson cohomology \cite{Va-94} of the vertical Poisson structure, \,$\cSch{\QH,P_{\beta}} \,=\, 0$.

Next, taking the interior product of both sides (as $(0,2)$-forms) of \ (\ref{C3}) with $\QV$, we get the following ``curvature identity'' \cite{YuV-2001}:
     \begin{equation*}
          \Curv^{\gamma}\big(u_{1},u_{2} \big) \,=\, -\,\pi^{\ast}\big( \omega(u_{1},u_{2}) \big)\,P_{\beta}^{\sharp}\,\dd\left(\tfrac{1}{\kappa}\right),
     \end{equation*}
for any \,$u_{1},u_{2} \in \X{B}$.\, This relation tells us that the scalar factor $\kappa$ controls the curvature of $\gamma$ in $M^{\Pi}$ in the following sense: the connection $\gamma$ is flat, \,$\Curv^{\gamma}=0$,\, if and only if $\kappa$ is a Casimir function of $P_{\beta}|_{M^{\Pi}}$ , i.e., condition (\ref{C5}) is satisfied in the coupling domain.

As we have already noted (see, Proposition \ref{PropFlatPair}), the flatness of $\gamma$ implies that the bivector field $\QH$ is a Poisson tensor. So, in the flat case, the horizonal $\Pi_{2,0}$ and vertical $\Pi_{0,2}$ parts of the almost coupling Poisson tensor $\Pi$ form a Poisson pair.

\begin{example}
Putting $\beta \equiv 0$, we get that $(\gamma,\kappa,0)$ is a Poisson triple for arbitrary flat connection $\gamma$ on the fiber bundle \,$\pi:M \rightarrow B$\, and an arbitrary smooth function \,$\kappa \in \Cinf{B}$. \hfill{$\triangleleft$}
\end{example}

Summarizing the above facts and taking into account the properties of coupling Poisson structures \cite{YuV-2001,Va-04,JVYu} we formulate the following result characterizing almost coupling structures in dimensions $5$. 
 \newpage
\begin{proposition}
An almost coupling Poisson tensor $\Pi$ (\ref{EP1}) on $M$ has the following behavior:
     \begin{itemize}
       \item[$(i)$] in the open subset \,$M^{\Pi} = M \setminus \mathcal{Z}(\kappa)$,\, $\Pi$ is a coupling Poisson tensor associated to the geometric data $(\gamma,\sigma,P_{\beta})$, where the Poisson connection $\gamma$ is uniquely determined by (\ref{AC2}) and the coupling form is given by \,$\sigma=\tfrac{1}{\kappa}\,\pi^{\ast}\omega$;

       \item[$(ii)$] in the open subset $\mathrm{Int}\big(\mathcal{Z}(\kappa)\big)$, the bivector field coincides with the vertical Poisson tensor, \,$\Pi=P_{\beta}$;

       \item[$(iii)$] at the points of the boundary $\partial\big(\mathcal{Z}(\kappa)\big)$, the linear part of $\kappa$ is compatible with $\beta$ by condition (\ref{C5}).
     \end{itemize}
\end{proposition}

As we know, the Poisson connection $\gamma$ is completely determined by $\Pi$ in the closure $\overline{{M}^{\Pi}}$ of the coupling domain. In the open subset $\mathrm{Int}\big(\mathcal{Z}(\kappa)\big)$, the connection $\gamma$ is not fixed and independent of $\beta$. Indeed, the freedom in the choice of $\gamma$ is given by transformation (\ref{CT1}), for arbitrary vector valued $1$-form $\Xi$ satisfying (\ref{CT2}) with \,$\mathrm{supp}\,\Xi \subset \mathrm{Int}(\mathcal{Z}(\kappa))$.\, But in the case when $M^{\Pi}$ is dense in $M$, the subset $\mathrm{Int}\big(\mathcal{Z}(\kappa)\big)$ is empty and hence the connection $\gamma$ is uniquely determined in the whole $M$.

\begin{remark}
Theorem \ref{TeoJacobiOrientable} and Theorem \ref{TeoACouptoTriple} remain true in the case when $M$ is a fibered oriented $(2+r)$-manifold over the $2$-dimensional base $B$, for any \,$r \geq 2$.\, In this case, every almost coupling Poisson tensor is defined by a triple $(\gamma,\kappa,\beta)$ satisfying the integrability conditions such that the only equation (\ref{IC1}) for $(r -2)$-form is modified. In particular, in the case \,$\dim{M}=4$,\, condition (\ref{IC1}) holds automatically and the integrability conditions for a smooth function \,$\beta \in \Cinf{M}$\, together with $\gamma$, $\kappa$ take the form \,$\kappa\,(\dd_{1,0}\beta-\beta\,\theta)=0$\, and \,$\beta\,\dd_{0,1}\kappa-\kappa^{2}\varrho=0$. \hfill{$\triangleright$}
\end{remark}

     \section{Unimodularity Criteria}

Here we describe some unimodularity criteria for almost coupling Poisson structures on $5$-dimensional fibered manifolds.

First, we recall that by the definition \cite{We-97}, the modular vector field $Z^{\Omega}$ of an oriented Poisson manifold $(M,\Pi)$ relative to a volume form $\Omega$ is defined by \,$\dlie{Z^{\Omega}}f \,:=\, \mathrm{div}_{\Omega}\big(\Pi^{\sharp}\dd{f}\big)$\, for \,$f \in \Cinf{M}$.\, As is known, $Z^{\Omega}$ is a Poisson vector field of $\Pi$ which is independent of the choice of a volume form modulo Hamiltonian vector fields. The Poisson structure $\Pi$ is said to be \emph{unimodular} if $Z^{\Omega}$ is a Hamiltonian vector field. In this case, there exists a volume form on $M$ which is invariant with respect to all Hamiltonian flows.

For example, a homogeneous Poisson structure $\Pi$ on the Euclidean space is unimodular if and only if the modular vector field of $\Pi$ relative to the Euclidean volume form is zero \cite{CIMP-94}.

Now, suppose we start again with an oriented fibered $5$-manifold $(\pi:M \rightarrow B,\Omega)$ with an oriented $2$-dimensional base $(B,\omega)$. Let $\Pi$ be the almost coupling Poisson tensor on $M$ associated to a Poisson triple $(\gamma,\beta,\kappa)$.

\begin{lemma}
The modular vector field $Z^{\Omega}$ of $\Pi$ on $M$ has the following bigraded decomposition \,$Z^{\Omega}=Z_{1,0}^{\Omega}+Z_{0,1}^{\Omega}$,\, where
     \begin{equation}
          Z_{1,0}^{\Omega} \,=\, -\mathbf{i}_{\kappa\theta \,+\, \dd_{1,0}\kappa}\,\hor^{\gamma}\psi \qquad \text{and} \qquad Z_{0,1}^{\Omega} \,=\, \mathbf{i}_{\dd_{0,1}\beta \,+\, \kappa\varrho}\QV. \label{BigradCampMod}
     \end{equation}
Here $1$-form $\theta$ and $2$-form $\varrho$ are given by formulas (\ref{F1}) and (\ref{F2}), respectively.
\end{lemma}
\begin{proof}
By the definition of $Z^{\Omega}$ we have \,$\mathbf{i}_{Z^{\Omega}}\Omega=-\dd{\mathbf{i}_{\Pi}}\Omega$.\, Moreover, for any $4$-forms $\delta$ and $\widetilde{\delta}$ of bidegree $(1,3)$ and $(2,2)$ respectively, the following equalities hold:
     \begin{equation*}
          \mathbf{i}_{\delta}\left(\QH\wedge \QV\right) \,=\, -\mathbf{i}_{\mathbf{i}_{\QV}\delta}\QH, \qquad \mathbf{i}_{\widetilde{\delta} }\left(\QH\wedge \QV\right) \,=\, \mathbf{i}_{\mathbf{i}_{\QH}\widetilde{\delta}}\QV.
     \end{equation*}
Applying these identities to the case when \,$\delta=-(\dd\mathbf{i}_{\Pi}\OmegaH) \wedge \OmegaV - (\mathbf{i}_{\Pi}\OmegaH)\,\dd_{1,0}\OmegaV$\, and \,$\widetilde{\delta}=-\OmegaH \wedge \dd{\mathbf{i}_{\Pi}}\OmegaV - (\mathbf{i}_{\Pi}\OmegaH)\,\dd_{2,-1}\OmegaV$,\, and taking into account (\ref{EP1})-(\ref{F2}), we verify the formulas for the bigraded components of the modular vector field.
\end{proof}

Consider the renormalization of the volume form by a nowhere vanishing function \,$a \in \Cinf{M}$:\, $a\Omega$.\, Then, the modular vector field and its bigraded components are changing by the rules \,$Z^{\,a\Omega}=Z^{\Omega}-\tfrac{1}{a}\,\mathbf{i}_{\dd{a}}\Pi$\, and
     \begin{align*}
          Z_{1,0}^{\,a\Omega} \,&=\, Z_{1,0}^{\Omega} \,+\, \kappa\,\mathbf{i}_{\dd_{1,0}a}\QH \,=\, \mathbf{i}_{\kappa\theta \,+\, \dd_{1,0}\kappa \,+\, \frac{\kappa}{a}\dd_{1,0}a}\QH, \\[0.15cm]
          Z_{0,1}^{\,a\Omega} \,&=\, Z_{0,1}^{\Omega} \,+\, \tfrac{1}{a}\,\mathbf{i}_{\dd_{0,1}a \wedge \beta}\QV \,=\, \mathbf{i}_{\dd_{0,1}\beta \,+\, \kappa\varrho \,+\, \frac{1}{a}\dd_{0,1}a \wedge \beta}\QV.
     \end{align*}

Now, putting \,$a=\frac{1}{\kappa}$\, and taking into account (\ref{IC3}), we conclude that, in the coupling domain \,$M^{\Pi}=M\setminus\mathcal{Z}(\kappa)$,\, the modular vector field $Z^{\Omega^{\prime}}$ relative to the volume form \,$\Omega^{\prime} := \frac{1}{\kappa}\,\Omega \,=\, \OmegaH \wedge \tfrac{1}{\kappa}\OmegaV$\, has the components
	\begin{equation}
		Z_{1,0}^{\Omega^{\prime}} \,=\, \kappa\,\mathbf{i}_{\theta}\QH \qquad \text{and} \qquad Z_{0,1}^{\Omega^{\prime}} \,=\, \mathbf{i}_{\dd_{0,1}\beta}\QV. \label{CamModPrima}
	\end{equation}
Assume that the vertical $1$-form \,$\beta \in \Gamma\HH^{\circ}$\, is $\dd_{0,1}$-closed,
     \begin{equation}
          \dd_{0,1}\beta \,=\, 0. \label{CL1}
     \end{equation}
It follows from here and (\ref{CamModPrima}) that the vertical component of the modular vector field relative to the volume form $\Omega'$ vanishes, \,$Z_{0,1}^{\Omega^{\prime}}=0$.

\begin{lemma}\label{LemaBeta01Cerrada}
Under assumption (\ref{CL1}) the vertical Poisson tensor $P_{\beta}$ is unimodular in $M$. Moreover, in the coupling domain \,$M^{\Pi}=M\setminus\mathcal{Z}(\kappa)$,\, the following conditions hold:
     \begin{itemize}
          \item the horizontal $1$-form \,$\theta \in \Gamma\V^{\circ}$\, in (\ref{F1}) takes values in the space $\mathrm{Casim}(M^{\Pi},P_{\beta})$ of a Casimir functions of the vertical Poisson structure;
          \item $\theta$ is covariantly constant,
               \begin{equation}
                    \dd_{1,0}\theta \,=\, 0. \label{CL3}
               \end{equation}
     \end{itemize}
\end{lemma}
\begin{proof}
First, the unimodularity of $P_{\beta}$ follows directly from (\ref{BigradCampMod}) and the closedness condition (\ref{CL1}). Is clear that the corresponding invariant volume form is given by (\ref{OmegaDef}). Next, by (\ref{C2}) and (\ref{CL1}) we have
     \begin{equation*}
          \dd_{0,1}(\beta \wedge \theta) \,=\, -\dd_{0,1} \circ \dd_{1,0}\beta \,=\, \dd_{1,0} \circ \dd_{0,1}\beta \,=\, 0,
     \end{equation*}
and hence \,$0=\dd_{0,1}(\beta\wedge\theta)=-\beta \wedge \dd_{0,1}\theta$.\, This implies that \,$\mathbf{i}_{\mathrm{hor}^{\gamma}u}\theta \in \mathrm{Casim}(M^{\Pi},P_{\beta})$,\, for any \,$u \in \X{B}$.\, Next,
using the identities \,$\dd_{0,1}\OmegaH=0$,\, $\dd_{0,1}\OmegaV=0$\, and $\dd_{0,1}\varrho=0$,\, by (\ref{F4}) we get
     \begin{align*}
          \dd_{1,0}\theta \wedge \OmegaV  \,&=\, \dd_{1,0}^{2}\OmegaV \ =\ -\dd_{2,-1} \circ \dd_{0,1}\OmegaV \,-\, \dd_{0,1} \circ \dd_{2,-1}\OmegaV \\
               \,&=\, -\dd_{0,1} \circ \dd_{2,-1}\OmegaV \,=\, -\dd_{0,1}\OmegaH \wedge \varrho \,-\, \OmegaH \wedge \dd_{0,1}\varrho \,=\, 0.
     \end{align*}
From here and the bigrading argument, we derive (\ref{CL3}).
\end{proof}

We arrive at the following result which is an almost-coupling version of a general unimodularity criterion for coupling Poisson structures due to \cite{AEYu-17}.

\begin{theorem}\label{TeoUnimodA-Coup}
Let $\Pi$ be an almost coupling Poisson tensor on $M$ associated to a Poisson triple $(\gamma,\beta,\kappa)$. Suppose that the closedness condition (\ref{CL1}) holds. Then,
	\begin{itemize}
		\item[$(i)$] in the coupling domain \,$M^{\Pi}=M\setminus\mathcal{Z}(\kappa)$,\, $\Pi$ is unimodular if and only if $\theta$ is $\dd_{1,0}$-exact in the sense that
				\begin{equation}
					\theta \,=\, -\dd_{1,0}h, \label{ThetaExact}
				\end{equation}
for a certain Casimir function \,$h \in \mathrm{Casim}(M^{\Pi},P_{\beta})$.\, The volume form on $M^{\Pi}$ given by
     \begin{equation}
          \Omega_{1}^{\mathrm{inv}} \,:=\, \tfrac{\mathrm{e}^{h}}{\kappa}\,\pi^{\ast}\omega\wedge\OmegaV \label{CH2},
     \end{equation}
is invariant with respect to the flows of all Hamiltonian vector fields of $\Pi$;

		\item[$(ii)$] in the open domain $\mathrm{Int}\,\mathcal{Z}(\kappa)$ of $M$, \,$\Pi=P_{\beta}$\, is unimodular and the corresponding invariant volume form is given by
		\begin{equation}
          \Omega_{2}^{\mathrm{inv}} \,:=\, \pi^{\ast}\omega\wedge\OmegaV \label{CH3}.
     \end{equation}
	\end{itemize}
\end{theorem}
\begin{proof}
Suppose that condition (\ref{ThetaExact}) holds. Consider the volume form $\Omega'$ which is well defined in $M^{\Pi}$. Then, it follows from (\ref{CamModPrima}) and (\ref{PVF1}) that the unimodular vector field of $\Pi$ with respect to $\Omega'$ is Hamiltonian, \,$Z^{\Omega'} = -\kappa\,\mathbf{i}_{\dd_{1,0}h}\QH = (X_{h})_{1,0} = X_{h}$\, and hence \,$\Omega_{1}^{\mathrm{inv}}=\mathrm{e}^{h}\Omega'$\, is an invariant volume form. The converse is also true. Indeed, the unimodularity of $\Pi$ implies that $a\Omega'$ is an invariant volume form for a certain nowhere vanishing function $a$ on $M^{\Pi}$. Then, by using again (\ref{CamModPrima}) one can show that $\theta$ is $\dd_{1,0}$-exact with primitive \,$h=\ln|a|$. Finally, the item (ii) follows from Lemma \ref{LemaBeta01Cerrada} and the equality \,$\Pi=P_{\beta}$\, in the subset $\mathcal{Z}(\kappa)$.
\end{proof}

\begin{remark}
The restriction of the covariant derivative $\dd_{1,0}$ to the spaces of horizontal forms with values in the space of Casimir functions of $P_{\beta}$ is well-defined and gives rise to a cochain complex, called the de Rham-Casimir complex \cite{AEYu-17}. So, Theorem \ref{TeoUnimodA-Coup} tells us that under assumption (\ref{CL1}), there exists a cohomological obstruction to the unimodularity property of the coupling Poisson structure $\Pi|_{M^{\Pi}}$ which is related with the generalized Reeb class \cite{AEYu-17}.
\end{remark}
Now we formulate the following global unimodularity criterion.

\begin{theorem}\label{TeoUnimodGlob}
Under hypothesis (\ref{CL1}), the Poisson structure \,$\Pi=\kappa\,\mathrm{hor}^{\gamma}\psi+P_{\beta}$\, is unimodular in the whole $M$ if and only if the condition (\ref{ThetaExact}) holds for a certain \,$h\in$ $\mathrm{Casim}(M^{\Pi},P_{\beta})$\, and there exists a nowhere vanishing function \,$K \in C_{M}^{\infty}$\, whose restriction to $\mathrm{Int}\mathcal{Z}(\kappa)$ is a Casimir function of $P_{\beta}$ and such that
     \begin{equation}
          \kappa|_{M^{\Pi}} \,=\, e^{h}K|_{M^{\Pi}} \label{KappaCasim}
     \end{equation}
Moreover, the volume form \,$\Omega^{\operatorname{inv}}=\frac{1}{K}\pi^{\ast}\omega\wedge\OmegaV$\, is invariant with respect to all Hamiltonian flows on $(M,\Pi)$.
\end{theorem}
\begin{proof}
The sufficiency part follows from the Theorem \ref{TeoUnimodA-Coup} and the relations \,$\Omega^{\mathrm{inv}}=\,\Omega^{\mathrm{inv}}_{1}$\, in \,$M^{\Pi}=M\setminus\mathcal{Z}(\kappa)$\, and \,$\Omega^{\mathrm{inv}}=K^{-1}\,\Omega^{\mathrm{inv}}_{2}$\, in $\mathrm{Int}\,\mathcal{Z}(\kappa)$,\, where $\Omega^{\mathrm{inv}}_{1}$ and $\Omega^{\mathrm{inv}}_{2}$ are defined by (\ref{CH2}) and (\ref{CH3}), respectively. To prove the necessity, suppose that $\Pi$ is unimodular in $M$. Then, there exists a nowhere vanishing \,$K \in C^{\infty}_{M}$\, such that \,$K^{-1}\,\pi^{\ast}\omega \wedge \OmegaV$\, is a (global) invariant volume form. On the other hand, by Theorem \ref{TeoUnimodA-Coup}, in the domain $M^{\Pi}$, the volume form (\ref{CH2}) is also invariant for a fixed primitive $h$ of $\theta$. Then, the above two invariants volume forms are related by a multiplicative Casimir factor \,$\kappa_{0} \in \mathrm{Casim}(M^{\Pi},\Pi)$,\, and hence,
	\begin{equation}\label{Kappa0Casim}
          \kappa|_{M^{\Pi}} \,=\, \mathrm{e}^{h}\kappa_{0}\,K|_{M^{\Pi}}
    \end{equation}
It follows from (\ref{CasimACoup}) that the primitive $h$ of $\theta$ is uniquely defined up to adding an arbitrary Casimir function of $\Pi|_{M^{\Pi}}$. We conclude that (\ref{Kappa0Casim}) is transformed to (\ref{KappaCasim}) under the changing \,$h \mapsto h+\ln|\kappa_{0}| \in \mathrm{Casim}(M^{\Pi},P_{\beta})$.\, Finally, the fact that \,$K \in \mathrm{Casim}(\mathrm{Int}\,\mathcal{Z}(\kappa),P_{\beta})$\, follows from the invariance of the volume forms $\Omega_{2}$ and $K^{-1}\Omega_{2}$ with respect to all Hamiltonian flows on $(\mathrm{Int}\,\mathcal{Z}(\kappa),P_{\beta})$.
\end{proof}

It follows from the proof of the Theorem \ref{TeoUnimodGlob} that condition (\ref{KappaCasim}) can be replaced by (\ref{Kappa0Casim}). A realization of condition (\ref{Kappa0Casim}) is given in Example \ref{EjmBR3Unimod}.

In general, if the zero set $\mathcal{Z}(\kappa)$ is not empty, then the factor $\tfrac{1}{\kappa}$ in (\ref{CH2}) has a singularity at \,$p \in \partial\big(\mathcal{Z}(\kappa)\big)$\, and the invariant volume form $\Omega_{1}^{\mathrm{inv}}$ is not necessarily to be extended to the whole $M$ by gluing with $\Omega_{2}^{\mathrm{inv}}$.

\begin{example}
In the $5$-dimensional Euclidean space \,$\mathbb{R}^{5}=\mathbb{R}_{x}^{2}\oplus\mathbb{R}_{y}^{3}$\, regarded as a trivial fiber bundle over $\mathbb{R}_{x}^{2}$, consider the following bivector field
     \begin{equation*}
          \Pi \,=\, \big(y_{1}^{2}-x_{1}^{2}-x_{2}^{2}\big)\left(\,  \frac{\partial}{\partial x_{1}} \wedge \frac{\partial}{\partial x_{2}} \,+\, \left( \frac{\partial}{\partial x_{1}} - \frac{\partial}{\partial x_{2}} \right) \wedge \left( \frac{\partial}{\partial y_{2}} + \frac{\partial}{\partial y_{3}}\right) \,\right) \,+\ y_{1}^{2}\,\frac{\partial}{\partial y_{2}}\wedge\frac{\partial}{\partial y_{3}}.
     \end{equation*}
Then, the bivector field $\Pi$ is an almost coupling Poisson tensor via the connection
     \begin{equation*}
          \gamma \,=\, \dd{y_{1}} \otimes \frac{\partial}{\partial y_{1}} \,+\, \dd{y_{2}} \otimes \frac{\partial}{\partial y_{2}} \,+\, \dd{y_{3}} \otimes \frac{\partial}{\partial y_{3}} \,-\, \dd(x_{1}+x_{2}) \otimes \left(\frac{\partial}{\partial y_{2}} + \frac{\partial}{\partial y_{3}}\right).
     \end{equation*}
In this case, \,$\beta=y_{1}^{2}\,\dd{y_{1}}$\, and \,$\theta=0$, and hence, formula (\ref{CH2}), gives an invariant volume form in the coupling domain \,$M^{\Pi}=\mathbb{R}^{5}\setminus\{y_{1}^{2}=x_{1}^{2}+x_{2}^{2}\}$.\, But, one can show that the function \,$\kappa=y_{1}^{2}-x_{1}^{2}-x_{2}^{2}$\, does not satisfy (\ref{KappaCasim}) and hence, by Theorem \ref{TeoUnimodGlob} the Poisson tensor $\Pi$ is not unimodular in the whole space $\mathbb{R}^{5}$. This fact is also derived from the observation that the unimodular vector field of the homogeneous Poisson tensor $\Pi$ with respect to the Euclidean volume form in $\mathbb{R}^{5}$ is nontrivial. \hfill{$\triangleleft$}
\end{example}

     \section{Symmetries of Integrability Conditions}

Here we described some symmetries of equations (\ref{IC1})-(\ref{IC3}), that is, some transformations which preserve the solutions of these equations. We start with so-called gauge transformations \cite{BuRa-03}. Our point is to describe a class of gauge transformations preserving the almost coupling property.

Suppose we are given a Poisson triple $(\gamma,\kappa,\beta)$ on an oriented fibered $5$-manifold $(\pi:M\rightarrow B,\Omega,\omega)$.

To formulate our results let us introduce the following notations. To any $1$-forms $\alpha_{1}$ and $\alpha_{2}$ on $M$, we assign a $2$-form \,$\{\alpha_{1}\wedge\alpha_{2}\}_{\beta}$\, on $M$ given by
     \begin{equation*}
          \{\alpha_{1}\wedge\alpha_{2}\}_{\beta}(Y_{1},Y_{2}) \,:=\, \{\alpha_{1}(Y_{1}),\alpha_{2}(Y_{2})\}_{\beta} \,-\, \{\alpha_{1}(Y_{2}),\alpha_{2}(Y_{1})\}_{\beta},
     \end{equation*}
for \,$Y_{1},Y_{2} \in \X{M}$.\, Recall that $\{,\}_{\beta}$ denotes the Poisson bracket on $M$ associated to the vertical Poisson $1$-form $\beta$, \,$\{f_{1},f_{2}\}_{\beta} = -\QV(\beta,\dd{f_{1}},\dd{f_{2}})$.

Suppose we are given a horizontal $1$-form \,$\mu \in \Gamma\V^{\circ}$\, and a Casimir function \,$c \in \Cinf{M}$\, of the bracket $\{,\}_{\beta}$, \,$\dd_{0,1}c \wedge \beta = 0$.\, Denote
     \begin{equation*}
          \varkappa_{\mu} \,=\, \varkappa_{\mu}(\gamma,\beta) \,:=\, \frac{\dd_{1,0}^{\gamma}\mu \,+\, \frac{1}{2}\{\mu \wedge \mu\}_{\beta}}{\OmegaH} \,\in\, \Cinf{M}.
     \end{equation*}
Here we use the fact that the nowhere vanishing $2$-form $\OmegaH$ and \,$\dd_{1,0}^{\gamma}\mu + \frac{1}{2}\{\mu \wedge \mu\}_{\beta}$\, are $2$-forms on $M$ of bidegree $(2,0)$. Now, we can associate to the pair $(\mu,c)$ a transformation
     \begin{equation}
          \mathcal{T}_{\mu,c}:\,(\gamma,\kappa,\beta) \,\longmapsto\, (\widetilde{\gamma},\widetilde{\kappa},\widetilde{\beta}) \label{TR1}
     \end{equation}
given by
     \begin{equation}
          \widetilde{\gamma}(X) \,=\, \gamma(X) \,+\, \mathbf{i}_{\dd_{0,1}\mu\wedge\beta}(\QV \wedge X), \qquad \widetilde{\kappa} \,=\, \frac{\kappa}{1-\kappa\,(\varkappa_{\mu}-c)}, \qquad \widetilde{\beta} \,=\, \beta, \qquad \label{TR2}
     \end{equation}
for any \,$X \in \X{M}$.

\begin{proposition}\label{PropTransfGauge}
Transformation (\ref{TR1})-(\ref{TR2}) preserves the solutions of equations (\ref{IC1})-(\ref{IC3}), that is, $(\widetilde{\gamma},\widetilde{\kappa},\widetilde{\beta})$ is again a Poisson triple.
\end{proposition}
\begin{proof}
First, let us start with a Poisson triple $(\gamma,\kappa,\beta)$ and an arbitrary triple $(\widetilde{\gamma},\widetilde{\kappa},\widetilde{\beta})$. Then, the connections $\gamma$ and $\widetilde{\gamma}$ are related by (\ref{CT1}), where a vector valued $1$-form \,$\Xi \in \Omega^{1}(M;\T{M})$\, satisfies the condition (\ref{CT2}). We observe that if the $1$-form \,$\widetilde{\beta} \in \Gamma\HH^{\circ}$\, is defined by
     \begin{equation}
          \widetilde{\beta} \,:=\, \beta-\Xi^{\ast}\beta, \label{CT3}
     \end{equation}
then $\widetilde{\beta}$ satisfies the condition \,$\widetilde{\beta}|_{\V}=\beta|_{\V}$\, and equation \,$\dd_{0,1}^{\widetilde{\gamma}}\widetilde{\beta} \wedge \widetilde{\beta}=0$.\, Here \,$\Xi^{\ast}:\T^{\ast}M \rightarrow \T^{\ast}M$\, is the adjoint vector bundle morphism. Now, assuming that
     \begin{equation}
          \Xi \,=\, -P_{\beta}^{\sharp} \circ (\dd_{0,1}\mu)^{\flat}, \label{CT4}
     \end{equation}
for a horizontal $1$-form \,$\mu \in \Gamma\V^{\circ}$,\, we conclude that condition (\ref{CT2}) is satisfied and the connection $\gamma$ is related with $\widetilde{\gamma}$ by (\ref{TR2}). Moreover, it follows from (\ref{CT4}) that \,$\Xi^{\ast}\beta=0$\, and hence, by equality (\ref{CT3}), the last relation in (\ref{TR2}) holds. Another consequence of (\ref{CT4}) is that, $\widetilde{\gamma}$ is a Poisson connection relative to $P_{\beta}$. This means that relation (\ref{IC2}) holds for the triple in (\ref{TR2}). Finally, the fact that the triple $(\widetilde{\gamma},\widetilde{\kappa},\widetilde{\beta})$ satisfies (\ref{IC3}) follows from the transition rule for the curvature of $\gamma$ under transformation (\ref{CT1}) and the relation \,$\mathbf{i}_{\mathrm{Curv}^{\gamma}(u_{1},u_{2})}\OmegaV=-\OmegaH(\hor^{\gamma}u_{1},\hor^{\gamma}u_{1})\,\varrho$,\, for all \,$u_{1},u_{2} \in \X{B}$.
\end{proof}

Now, let \,$\Pi=\Pi_{2,0}+\Pi_{0,2}$\, and \,$\widetilde{\Pi}=\widetilde{\Pi}_{2,0}+\widetilde{\Pi}_{0,2}$\, be two almost coupling Poisson structures associated with some Poisson triples $(\gamma,\kappa,\beta)$ and $(\widetilde{\gamma},\widetilde{\kappa},\widetilde{\beta})$, respectively. Suppose that these triples are related by a transformation (\ref{TR1}) for a certain pair $(\mu,c)$. Then, \,$\Pi_{0,2}=\widetilde{\Pi}_{0,2}=P_{\beta}$\, because of \,$\beta=\widetilde{\beta}$\, and hence the transformation \,$\Pi \mapsto \widetilde{\Pi}$\, modifies only the horizontal part, \,$\widetilde{\Pi}_{2,0}=-\widetilde{\kappa}\,\widetilde{Q}_{H}$.\, It follows \ from here and the second relation in (\ref{TR2}) that the domain of definition of the transformed Poisson structure $\widetilde{\Pi}$ is the following open subset in $M$:
     \begin{equation*}
          \mathrm{Dom}\big(\widetilde{\Pi}\big) \,:=\, \{\, p \in M \,|\, 1-\kappa(p)(\varkappa_{\mu}(p)-c(p)) \,\neq\, 0 \,\}.
     \end{equation*}
Moreover, \,$\mathcal{Z}(\widetilde{\kappa})=\mathcal{Z}(\kappa)$.\, Taking into account the transition rule for the horizontal lift, \,$\mathrm{hor}^{\widetilde{\gamma}}{u}  = \mathrm{hor}^{\gamma}{u} \,+\, P^{\sharp}\dd\mu(u)$,\, we conclude that the characteristic distributions of $\Pi$ and $\widetilde{\Pi}$ coincide, \,$C^{\Pi}=C^{\widetilde{\Pi}}$.\, Let $(\mathcal{S},\varpi)$ and $(\mathcal{S},\widetilde{\varpi})$ be the symplectic foliations on $\mathrm{Dom}(\widetilde{\Pi})$ carrying the leaf-wise symplectic forms $\varpi$ and $\widetilde{\varpi}$ of $\Pi$ and $\widetilde{\Pi}$, respectively. Then, one can show that the difference between $\widetilde{\varpi}$ and $\varpi$ is the pull-back to $\mathcal{S}$ of the global $2$-form \,$\Upsilon:=-\dd\mu + c\,\OmegaH$\, on $M$. In other words, for every leaf \,$\iota_{S}:S \hookrightarrow E$\, of the characteristic foliation $\mathcal{S}$, we have \,$\widetilde{\varpi}_{S}-\varpi_{S}=\iota_{S}^{\ast}\Upsilon$.\, Therefore, $\widetilde{\Pi}$ is a result of the gauge transformation of the Poisson structure $\Pi$ via the $2$-form $\Upsilon$ which is closed along the leaves of $\mathcal{S}$. The closedness of $\Upsilon$ is equivalent to the condition \,$c \in \pi^{\ast}\Cinf{B}$.\, Note that the complement \,$M \setminus \mathrm{Dom}(\widetilde{\Pi})$\, just consists of all points \,$p \in M$\, at which the $2$-form \,$\iota_{S}^{\ast}\Upsilon+\varpi_{S}$\, becomes degenerate, where $S$ is the leaf through $p$.

Now, consider another simple symmetry of equations (\ref{IC1})-(\ref{IC3}) defined by \,$\rho_{\varepsilon}:(\gamma,\kappa,\beta)\mapsto(\gamma,\varepsilon\kappa,\varepsilon\beta)$\, for any fixed \,$\varepsilon\in\mathbb{R}$.\, Clearly, \,$\Pi \mapsto \widetilde{\Pi}=\varepsilon\Pi$.

Starting with an almost coupling Poisson tensor on the oriented fibered $5$-manifold $(M,\Omega)$ over the oriented $2$-manifold $(B,\omega)$ and using the above symmetries, we get the following recipe to construct a ``new'' Poisson structure from the original one.

\begin{proposition}
For a given arbitrary pair $(\mu,c)$ and an almost coupling Poisson tensor $\Pi$ on $M$ associated with a Poisson triple $(\gamma,\kappa,\beta)$, the symmetry transformation
     \begin{equation}
          \rho_{\frac{1}{\varepsilon}} \circ \mathcal{T}_{\mu,c} \circ \rho_{\varepsilon} \label{GT1}
     \end{equation}
induces the following smooth $\varepsilon$-dependent family \,$\{\Pi_{\varepsilon}\}_{\varepsilon\in\mathbb{R}}$\, of Poisson structures:
     \begin{equation}
          \Pi_{\varepsilon} \,=\, \frac{\kappa}{1 \,-\, \varepsilon\kappa\,(\varkappa_{\mu,\varepsilon}-c)}\,\mathrm{hor}^{\gamma_{\varepsilon}}\psi \,-\, \mathbf{i}_{\beta}\QV, \label{GT2}
     \end{equation}
where \,$\gamma_{\varepsilon}(X):=\gamma(X)+\varepsilon\,\mathbf{i}_{\dd_{0,1}\mu \wedge \beta}(\QV \wedge X)$,\, for any \,$X \in \X{M}$,\, and
     \begin{equation*}
         \varkappa_{\mu,\varepsilon} \,:=\, \frac{\dd_{1,0}^{\gamma}\mu \,+\, \frac{\varepsilon}{2}\,\{\mu\wedge\mu\}_{\beta}}{\pi^{\ast}\omega}.
     \end{equation*}
\end{proposition}
\begin{proof}
Applying transformation (\ref{GT1}) to the triple $(\gamma,\kappa,\beta)$, we get the Poisson triple \\ $\left( \widetilde{\gamma} = \gamma_{\varepsilon}, \,\widetilde{\kappa} = \frac{\kappa}{1-\varepsilon\kappa\,(\varkappa_{\mu,\varepsilon}-c)}, \, \widetilde{\beta} = \beta \right)$ inducing the Poisson tensor $\Pi_{\varepsilon}$ in (\ref{GT2}). \mbox{\quad}
\end{proof}

\begin{corollary}
If $U$ is an open subset in $M$ with compact closure, then there exists a \,$\delta > 0$\, such that for the Poisson tensor $\Pi_{\varepsilon}$ in (\ref{GT2}) we have \,$U  \subseteq \mathrm{Dom}(\Pi_{\varepsilon})$,\, for all \,$\varepsilon \in(-\delta,\delta)$.
\end{corollary}

In particular, if $M$ is \emph{compact}, then for small enough $\varepsilon$, the Poisson tensor $\Pi_{\varepsilon}$ is well-defined in the whole $M$ and can be regarded as a deformation of the original one, \,$\Pi=\Pi_{\varepsilon}|_{\varepsilon=0}$.\, Moreover, by using the Moser homotopy method, one can show that if $c=0$, then the Poisson structures $\Pi$ and $\Pi_{\varepsilon}$ are isomorphic \cite{MYu,AEYu-17}.

As we noted in Section 3, the fiber preserving transformations respect the almost coupling property on fiber bundles. In other words, we observe that there exists a natural action of fiber preserving diffeomorphisms \,$g:M \rightarrow M$\, on the set of Poisson triples which is given by \,$g^{\ast}(\gamma,\kappa,\beta):=(g^{\ast}\gamma,g^{\ast}\kappa,g^{\ast}\beta)$.\, Therefore, the pull-back $g^{\ast}\Pi$ of any almost coupling Poisson tensor $\Pi$ associated with a Poisson triple $(\gamma,\kappa,\beta)$ is again an almost coupling Poisson tensor on $M$ whose Poisson triple is just $g^{\ast}(\gamma,\kappa,\beta)$. One can verify this fact directly by applying the pull back $g^{\ast}$ to both sides of equations (\ref{IC1})-(\ref{IC3}).

     \section{The Case of Trivial Fiber Bundles}

To construct a family of Poisson structures by using formula (\ref{CT2}), we need to fix an original almost coupling Poisson structure $\Pi$ or, a ``particular'' solution to equations (\ref{IC1})-(\ref{IC3}). According to results of Section $4$, one can fix $\Pi$ in the following situation.

\begin{proposition}
Let $(M\overset{\pi}{\rightarrow}B,P_{\beta})$ be a flat Poisson fiber bundle over a symplectic $2$-dimensional base $(B,\omega)$, which is equipped with a flat Poisson connection $\gamma$. Then, for any Casimir function \,$\kappa_{0} \in \mathrm{Casim}(M;P_{\beta})$,\, the triple $(\gamma,\kappa_{0},\beta)$ is Poisson and induces the Poisson tensor \,$\Pi=\kappa_{0}\,\mathrm{hor}^{\gamma}\psi + P_{\beta}$.
\end{proposition}
\begin{proof}
The condition for the connection to be Poisson is equivalent to the relation \,$\dd_{1,0}\beta + \beta \wedge \theta = 0$\,  on M which implies (\ref{IC2}). The flatness of $\gamma$ means that \,$\varrho=0$\, and the equality \,$\dd_{0,1}\kappa_{0} \wedge \beta = 0$\, holds for any Casimir function $\kappa_{0}$ of $P_{\beta}$. This proves (\ref{IC3})
\end{proof}

In this section, we consider the case when the flat Poisson bundle is trivial and comes from the product of two Poisson manifolds. Let $(B,\omega)$ be a symplectic $2$-manifold and $(N,\Omega^{\mathrm{fib}})$ an orientable $3$-manifold equipped with a Poisson tensor
$P_{\mathrm{fib}}$ and a volume form $\Omega^{\mathrm{fib}}$. Then, one can choose adapted local coordinate systems \,$x=(x^{1},x^{2})$\, on $B$ and \,$y=(y^{1},y^{2},y^{3})$\, on $N$ such that \,$\omega = \dd{x^{1}} \wedge \dd{x^{2}}$\, and \,$\Omega^{\mathrm{fib}} \,=\, \dd{y^{1}} \wedge \dd{y^{2}} \wedge \dd{y^{3}}$.

Consider the product manifold \,$M=B \times N$\, and denote by \,$\mathrm{pr}_{1}:M \rightarrow B$\, and \,$\mathrm{pr}_{2}:M \rightarrow N$\, the canonical projections. Then, we have a trivial bundle \,$\pi=\mathrm{pr}_{1}:M \rightarrow B$\, over $B$ with typical fiber $N$. The pair $(\omega,\Omega^{\mathrm{fib}})$ induces a volume form on the total space $E$ given by
     \begin{equation}
          \Omega \,=\, \pi^{\ast}\omega\wedge \mathrm{pr}_{2}^{\ast}\Omega^{\mathrm{fib}}. \label{PR2}
     \end{equation}
Let \,$\gamma=\eta^{a} \otimes \frac{\partial}{\partial y^{a}}$\, be an Ehresmann connection on $M$, where \,$\eta^{a}=\dd{y^{a}}+\gamma_{i}^{a}(x,y)\,\dd{x^{i}}$,\, $a=1,2,3$.\, Then, the corresponding horizontal subbundle $\HH$ is
generated by vector fields \,$\mathrm{hor}_{i}^{\gamma}=\frac{\partial}{\partial x^{i}}-\gamma_{i}^{a}\,\frac{\partial}{\partial y^{a}}$,\, $i=1,2$.\, It follows from (\ref{PP1}) and (\ref{PR2}) that there exists a unique horizontal $3$-form \,$\OmegaV \in \Gamma\wedge^{3}\HH^{0}$\, such that \,$\Omega=\pi^{\ast}\omega \wedge \OmegaV$.\, Locally, \,$\OmegaV=\eta^{1} \wedge
\eta^{2} \wedge \eta^{3}$.

Let $(\gamma,\beta,\kappa)$ be a Poisson triple on the trivial bundle \,$M=B\times N$.\, In adapted coordinates $(x,y)$, we have \,$\beta=\beta_{a}(x,y)\,\eta^{a}$\, and
     \begin{equation}
          \theta \,=\, -\frac{\partial\gamma_{i}^{a}}{\partial{y^{a}}}\,\dd{x^{i}}, \qquad \varrho \,=\, -\tfrac{1}{2}\,\epsilon_{abc}\varrho^{a}(x,y)\,\eta^{b} \wedge \eta^{c}, \label{M2}
     \end{equation}
where \,$\varrho^{a} \,=\, \frac{\partial\gamma_{1}^{a}}{\partial x^{2}} \,-\, \frac{\partial\gamma_{2}^{a}}{\partial{x^{1}}} \,+\, \gamma_{1}^{b}\,\frac{\partial \gamma_{2}^{a}}{\partial y^{b}} \,-\, \gamma_{2}^{b}\,\frac{\partial\gamma_{1}^{a}}{\partial y^{b}}$\, and $\epsilon_{abc}$ denotes the totally anti-symmetric Levi-Civita symbol.

In terms of the components of $\gamma$ and $\beta$, equations (\ref{IC1})-(\ref{IC3}) take the form
     \begin{align*}
          \underset{(a,b,c)}{\mathfrak{S}}\left(\frac{\partial\beta_{a}}{\partial y^{b}}-\frac{\partial\beta_{b}}{\partial y^{a}} \right)\beta_{c} \,&=\, 0, \\[0.15cm]
          \kappa\,\left( \frac{\partial\beta_{a}}{\partial x^{i}} \,-\, \gamma_{i}^{b}\,\frac{\partial\beta_{a}}{\partial y^{b}} \,-\, \beta_{b}\,\frac{\partial\gamma_{i}^{b}}{\partial y^{a}} \,+\, \beta_{a}\,\frac{\partial\gamma_{i}^{b}}{\partial y^{b}} \,\right) \,&=\, 0, \\[0.15cm]
          \frac{\partial\kappa}{\partial y^{a}}\,\beta_{b} \,-\, \frac{\partial\kappa}{\partial y^{b}}\beta_{a} \,+\, \epsilon_{abc}\kappa^{2}\,\varrho^{c} \,&=\, 0.
     \end{align*}

Now, let us assume that the $3$-manifold $N$ is also equipped with a Poisson structure $P_{\mathrm{fib}}$ which is induced by a Poisson $1$-form \,$\beta^{\mathrm{fib}}=\beta_{a}(y)\,\dd{y^{a}}$,\, that is, $\mathbf{i}_{P_{\mathrm{fib}}}\Omega^{\mathrm{fib}}=\beta^{\mathrm{fib}}$.\, Equivalently, \,$P_{\mathrm{fib}}=\mathbf{i}_{\beta^{\mathrm{fib}}}\frac{\partial}{\partial y^{1}} \wedge \frac{\partial}{\partial y^{2}} \wedge \frac{\partial}{\partial y^{3}}$.\, Therefore, $(N,P_{\mathrm{fib}})$ is a typical fiber of the trivial Poisson bundle \,$\pi:M \rightarrow B$.\, Suppose that we are given a setup $(\kappa_{0},c,\mu)$ consisting of some fiberwise Casimir functions \,$\kappa_{0},c \in \Cinf{M}$,\, i.e., \,$\kappa_{0},c\,|_{\{x\}\times N} \in \mathrm{Casim}(N,P_{\mathrm{fib}})$,\, for all \,$x \in B$,\, and a global horizontal $1$-form \,$\mu=\mu_{i}(x,y)\,\dd{x^{i}}$\, on $M$.

We associate to the data $(\kappa_{0},c,\mu,\beta^{\mathrm{fib}})$ the $\varepsilon$-dependent triple $(\gamma_{\varepsilon},\kappa_{\varepsilon},\beta=\mathrm{pr}_{2}^{\ast}\beta^{\mathrm{fib}})$ given by
     \begin{equation}
          (\gamma_{\varepsilon})_{i}^{a} \,=\, \varepsilon\,\epsilon^{abc}\frac{\partial\mu_{i}}{\partial y^{b}}\,\beta_{c}, \qquad \kappa_{\varepsilon} \,=\, \frac{\kappa_{0}}{1-\varepsilon\kappa_{0}\,(\varkappa_{\mu}-c)}, \qquad \beta \,=\, \beta_{a}(y)\,\eta^{a}. \quad \label{T1}
     \end{equation}

Let \,$\gamma=\dd{y^{a}} \otimes \frac{\partial}{\partial y^{a}}$\, be the flat Ehresmann connection on \,$M=B \times N$\, associated to the canonical horizontal distribution \,$\HH(x,y)=\T_{x}B \oplus\{0\}$.\, Clearly, $(\gamma,\kappa_{0},\beta)$ is a Poisson triple. Then, the triple $(\gamma_{\varepsilon},\kappa_{\varepsilon},\beta)$ in (\ref{T1}) is obtained by applying the gauge transformation (\ref{GT1}) to $(\gamma,\kappa_{0},\beta)$ and hence, by Proposition \ref{PropTransfGauge}, is again a solution to equations (\ref{IC1})-(\ref{IC3}).

So, we conclude that the original setup $(\kappa_{0},c,\mu,\beta^{\mathrm{fib}})$ induces the following $\varepsilon$-dependent Poisson
bivector field
     \begin{equation}
          \Pi_{\varepsilon} \,=\, \frac{\kappa_{0}}{1-\varepsilon\kappa_{0}\,(\varkappa_{\mu}-c)}\,\mathrm{hor}_{1}^{\gamma_{\varepsilon}} \wedge \mathrm{hor}_{2}^{\gamma_{\varepsilon}} \,+\, \epsilon^{abc}\beta_{a}(y)\,\frac{\partial}{\partial y^{b}}\wedge\frac{\partial}{\partial y^{c}}. \label{N2}
     \end{equation}
Here \,$\mathrm{hor}_{i}^{\gamma_{\varepsilon}}=\frac{\partial}{\partial x^{i}}-\varepsilon\,\epsilon^{abc}\,\frac{\partial\mu_{i}}{\partial y^{a}}\,\beta_{b}\,\frac{\partial}{\partial y^{c}}$ \ and \ $\varkappa_{\mu,\varepsilon} \,=\, \frac{\partial\mu_{2}}{\partial x^{1}} \,-\, \frac{\partial\mu_{1}}{\partial x^{2}} \,-\, \varepsilon\,\epsilon^{abc}\,\frac{\partial\mu_{1}}{\partial y^{a}}\,\beta_{b}\,\frac{\partial\mu_{2}}{\partial y^{c}}$.

Suppose that there exists a singular point \,$y_{0} \in N$\, of \,$P_{\mathrm{fib}}$, \,$\beta_{b}(y_{0})=0$\, $(b=1,2,3)$\, such that
\,$\mu_{1}(x,y_{0})=\mu_{2}(x,y_{0})=0$\, and \,$c(x,y_{0})=0$\, for all \,$x \in B$.\, Then, for a fixed $\varepsilon$, the bivector field $\Pi_{\varepsilon}$ (\ref{N2}) is well defined in a neighborhood $U_{\varepsilon}$ of the section $B \times \{y_{0}\}$ which  represents a $2$-dimensional Poisson submanifold of $(U_{\varepsilon},\Pi_{\varepsilon})$ equipped with the Poisson structure \,$\Psi=\kappa_{0}(x,y_{0})\,\frac{\partial}{\partial x^{1}}\wedge\frac{\partial}{\partial x^{2}}$\, (see Proposition \ref{PropSubPoisson}).

Now, let us apply the unimodularity criteria in Theorem \ref{TeoUnimodGlob} to the family (\ref{N2}). Assume that the $1$-form $\beta^{\mathrm{fib}}$ is closed, \,$\frac{\partial\beta_{a}}{\partial y^{b}}=\frac{\partial\beta_{b}}{\partial y^{a}}$.\, This implies that $P_{\mathrm{fib}}$ is unimodular. Then, we claim that the Poisson bivector field $\Pi_{\varepsilon}$ (\ref{N2}) is unimodular in the domains \,$M^{\Pi}=M\setminus\mathcal{Z}(\kappa)$\, and $\mathrm{Int}\,\mathcal{Z}(\kappa)$. Indeed, by using formulas (\ref{M2}) we compute
               \begin{equation*}
                    \theta_{i} \,=\, -\frac{\partial(\gamma_\varepsilon)_{i}^{a}}{\partial y^{a}} \,=\, \varepsilon\,\frac{\partial^{2}\mu_{i}}{\partial y^{a}\partial y^{b}}P_{\mathrm{fib}}^{ba} \,+\, \varepsilon\,\frac{\partial\mu_{i}}{\partial y^{b}}\frac{\partial P_{\mathrm{fib}}^{ba}}{\partial y^{a}}.
               \end{equation*}
The first term in the right hand side of this equality is zero. The second one vanishes because of the closedness condition, \,$\frac{\partial P_{\mathrm{fib}}^{bc}}{\partial y^{c}}=\epsilon^{abc}\,\frac{\partial\beta_{a}}{\partial y^{c}}=0$.\, Consequently, \,$\theta\equiv0$.\, So, (\ref{ThetaExact}) holds for \,$h \equiv 0$.\, Moreover, $\Pi_{\varepsilon}$ satisfies the item (ii) of Theorem \ref{TeoUnimodA-Coup} and the corresponding invariant volume form is given by (\ref{PR2}).

Here is a realization of the global criteria in Theorem \ref{TeoUnimodGlob}.

\begin{example}\label{EjmBR3Unimod}
Let \,$M=B \times \mathbb{R}^{3}$,\, where $B$ is compact. Consider \,$N=\mathbb{R}^{3}$\, equipped with cyclic brackets associated with the Lie-Poisson bracket on $\mathfrak{so}^{\ast}(3)$. Put \,$\beta^{\mathrm{fib}}=y \cdot dy$\, and \,$\kappa_{0}(y)=\chi(\| y \|^{2})$,\, where \,$\chi \in \Cinf{\mathbb{R}}$\, is a bump function with \,$\mathrm{supp}\chi=[0,1]$.\, Then, for arbitrary $c,\mu$ and sufficiently small $\varepsilon$, formula (\ref{N2}) gives the almost coupling Poisson tensor $\Pi_{\varepsilon}$ defined on the whole $M$ with the coupling domain \,$M^{\Pi_{\varepsilon}}=B \times \{\| y \|<1\}$.\, In this case, the all hypothesis of Theorem \ref{TeoUnimodGlob} hold and condition (\ref{Kappa0Casim}) is satisfied for \,$h=0$\, and \,$K=\big(1-\varepsilon\,\kappa_{0}(\varkappa_{\mu}-c)\big)^{-1}$.\, Therefore, $\Pi_{\varepsilon}$ is unimodular on $M$ and the
corresponding global invariant volume form is given by
     \begin{equation*}
          {\Omega}^{\mathrm{inv}}_{\varepsilon} \,=\, \big(1-\varepsilon\,\kappa_{0}(\varkappa_{\mu}-c)\big)\,\pi^{\ast}\omega \wedge \mathrm{pr}^{\ast}_{2}\Omega^{\mathrm{fib}}.
     \end{equation*}
\hfill{$\triangleleft$}
% If $B$ is connected, then the zero section $B\times\{0\}$ is a $2$-dimensional symplectic leaf of $\Pi_{\varepsilon}$.
\end{example}

%Finally, we consider the family of Poisson structures $\Pi_{\varepsilon}$ (\ref{N2}) in the context of deformation theory.

%\begin{example}
%Assume that \,$\kappa_{0}=\pi^{\ast}k_{0}$\, for a certain \,$k_{0} \in \Cinf{B}$.\, Then, the base $B$ can be regarded as a Poisson $2$-manifold equipped with Poisson tensor \,$\Psi=k_{0}(x)\,\frac{\partial}{\partial x^{1}} \wedge \frac{\partial}{\partial x^{2}}$.\, It is clear that the zeros of $k_{0}$ are just the points of zero rank of $\Psi$. Therefore, the ``limiting'' Poisson tensor $\Pi_{0}$ is the product of two Poisson structures $\Psi$ and $P_{\mathrm{fib}}$. For small $\varepsilon$, the Poisson bivector $\Pi_{\varepsilon}$ is viewed as a deformation of the product Poisson structure $\Pi_{0}$.
%\end{example}

%\textit{Poisson Submanifolds.} Suppose that there exists a zero point \,$y_{0} \in N$\, of $P_{\mathrm{fib}}$, \,$\beta_{b}(y_{0}) = 0$\, \,$(b=1,2,3)$\, such that \,$\mu_{1}(x,y_{0})=\mu_{2}(x,y_{0})=0$\, and \,$c(x,y_{0})=0$,\, for all \,$x \in B$.\, Then, for a fixed $\varepsilon$, the bivector field $\Pi_{\varepsilon}$ (\ref{N2}) is well defined in a neighborhood $U$ of the section \,$B \times\{y_{0}\}$\, which represents a $2$-dimensional Poisson submanifold of $(U,\Pi_{\varepsilon})$ equipped with Poisson structure \,$\Psi=\kappa_{0}(x,y_{0})\,\frac{\partial}{\partial x^{1}} \wedge \frac{\partial}{\partial x^{2}}$.

     \section*{Acknowledgments}

This work was partially supported by the Mexican National Council of Science and Technology (CONACyT), under research project CB-$2013$-$219631$.
%\appendix

%\section{Appendices}

%Use only when absolutely necessary. They
%should come before the References. If there is more than one
%appendix, number them alphabetically. Number displayed equations
%occurring in the Appendix as follows: (\ref{app1}), (A.2),
%etc.
%\begin{equation}
%\mu(n, t) = \frac{\sum^\infty_{i=1} 1(d_i < t,
%N(d_i) = n) }{\int^t_{\sigma=0} 1(N(\sigma) = n)d\sigma}\,.
%\label{app1}
%\end{equation}

     %\section*{References}

% References are to be listed in the order cited in the text in Arabic numerals within square brackets. They can be referred to indirectly, e.g.~``$\ldots$ in the statement \cite{beeson}.'' or used directly, e.g.~``$\ldots$ see [2] for examples.'' List references using the style shown in the following examples. For journal names, use the standard abbreviations.  Typeset references in 9 pt roman.

\pdfbookmark[1]{References}{ref}
\LastPageEnding

\end{document}